\documentclass[12pt]{amsart}

\usepackage{amsmath,amsfonts,amsthm,amscd,amssymb,mathrsfs}
\newtheorem{theorem}{Theorem}

\newtheorem{proposition}[theorem]{Proposition}
\newtheorem{lemma}[theorem]{Lemma}
\newtheorem{definition}[theorem]{Definition}

\newtheorem{corollary}[theorem]{Corollary}

\theoremstyle{remark}
\newtheorem{remark}[theorem]{Remark}

\newcommand{\fr}[2]{\hbox{\large$\frac{#1}{#2}$}}

\renewcommand{\tt}{\tau}

\numberwithin{equation}{section}
\numberwithin{theorem}{section}
\begin{document}
\bibliographystyle{amsalpha} 
\title[Orthogonal complex
structures]{Orthogonal complex structures on domains in $\mathbb{R}^4$} 

\author{Simon Salamon}
\thanks{Partially supported by MIUR (Metriche Riemanniane e
   Variet\`a Differenziabili, PRIN\,05)}
\address{Simon Salamon, Dipartimento di Matematica, Politecnico di
  Torino, Corso Duca degli Abruzzi 24, 10129 Torino, Italy.}
\email{simon.salamon@polito.it}
\author{Jeff Viaclovsky}
\address{\vspace{-2mm}Jeff Viaclovsky, Department of Mathematics, MIT, Cambridge, MA 02139}
\address{Department of Mathematics, University of Wisconsin, Madison, 
WI, 53706}
\email{jeffv@math.wisc.edu}
\thanks{Research partially supported by NSF Grant DMS-0503506}
\dedicatory{To Nigel Hitchin on the occasion of his 60th birthday}
\begin{abstract}
 An orthogonal complex structure on a domain in $\mathbb{R}^4$ is a complex structure 
which is integrable and is compatible with the Euclidean metric.
This gives rise to a first order system of partial differential 
equations which is conformally invariant.
We prove two Liouville-type uniqueness theorems for solutions
of this system, and use these to give an alternative proof of 
the classification of compact locally conformally flat 
Hermitian surfaces first proved by Pontecorvo.
We also give a classification of non-degenerate
quadrics in $\mathbb{CP}^3$ under the action of the conformal group
$SO_\circ(1,5)$. Using this classification, we show that generic  
quadrics give rise to orthogonal complex structures defined on the 
complement of unknotted solid tori which are smoothly embedded in $\mathbb{R}^4$.
\end{abstract}
\date{April 8, 2008}
\maketitle
\newcommand{\pp}{\partial / \partial}
\newcommand{\pzj}{\partial / \partial z_j}
\newcommand{\pzbj}{\partial / \partial \bar{z}_j}
\newcommand{\pxj}{\partial / \partial x_j}
\newcommand{\pyj}{\partial / \partial y_j}
\newcommand{\dzj}{dz_j}
\newcommand{\dbzj}{d \bar{z}_j}
\newcommand{\la}{\langle}
\newcommand{\ra}{\rangle}
\newcommand{\bz}{\bar{z}}
\newcommand{\bw}{\bar{w}}
\newcommand{\bo}{\bar{1}}
\newcommand{\bt}{\bar{2}}
\newcommand{\xa}{\xi_0}
\newcommand{\bxa}{\overline{\xi}_0}
\newcommand{\xb}{\xi_{12}}
\newcommand{\bxb}{\overline{\xi}_{12}}
\newcommand{\cc}{\mathbb{C}}
\newcommand{\hh}{\mathbb{H}}
\newcommand{\rr}{\mathbb{R}}
\newcommand{\bc}{\bar{c}}
\newcommand{\cp}{\mathbb{CP}}
\newcommand{\Wa}{\xi_0 z_1 - \xi_{12} \bar{z}_2}
\newcommand{\Wb}{\xi_0 z_2 + \xi_{12} \bar{z}_1}
\newcommand{\vs}{\vskip10pt\medbreak\noindent}
\newcommand{\ba}{\begin{array}}
\newcommand{\ea}{\end{array}} 
\newcommand{\be}[1]{\begin{equation}\label{#1}}
\newcommand{\ee}{\end{equation}}
\newcommand{\rf}[1]{(\ref{#1})}
\newcommand{\ab}{{\alpha\beta}}
\newcommand{\C}{\mathbb{C}}
\newcommand{\HH}{\mathbb{H}}
\newcommand{\R}{\mathbb{R}}
\newcommand{\ol}{\overline}
\def\bar{\ol}
\newcommand{\vv}{\mathbf{v}}
\newcommand{\ww}{\mathbf{w}}
\newcommand{\xx}{\mathbf{x}}
\newcommand{\yy}{\mathbf{y}}
\newcommand{\zz}{\mathbf{z}}
\newcommand{\tK}{\tilde\Lambda\kern1pt}
\newcommand{\tQ}{\tilde Q\kern1pt}
\newcommand{\tp}{^{\kern-1pt\top}\kern-2pt}
\newcommand{\mat}[1]{\left(\!\ba{cc}#1\ea\!\right)}
\newcommand{\Ext}{\raise1pt\hbox{$\bigwedge$}}
\renewcommand{\Im}{\mathop{\mathfrak{Im}}}
\renewcommand{\Re}{\mathop{\mathfrak{Re}}}
\newcommand{\block}[2]{\big\{#1\kern2pt|\kern2pt#2\big\}}
\newcommand{\bxi}{\hbox{\boldmath$\xi$}}
\newcommand{\FF}{F}
\newcommand{\K}{K}
\newcommand{\hyp}{\mathscr{H}}
\newcommand{\Q}{\mathscr{Q}}
\renewcommand{\ge}{\geqslant}
\renewcommand{\le}{\leqslant}
\renewcommand{\l}{\lambda}
\newcommand{\m}{\mu}
\newcommand{\n}{\nu}
\newcommand{\rank}{\mathrm{rank}\,}
\parskip1pt

\setcounter{tocdepth}{1}
\vspace{-5mm}
\tableofcontents

\section{Introduction}

Let $(M^4,g)$ denote a $4$-dimensional oriented Riemannian manifold.
\begin{definition}
  An {\em{almost complex structure}} is an endomorphism $J : TM
  \rightarrow TM$ satisfying $J^2 = - I$.  The almost complex
  structure $J$ is said to be {\em{orthogonal}} if it is an orthogonal
  transformation, that is,
\begin{align*}
g(J v, J w) = g(v, w)
\end{align*}
for every $v, w \in T_pM,$ and preserves orientation. 
An orthogonal almost complex structure is said to 
be an {\em{orthogonal complex structure}} or an {\em{OCS}} 
if $J$ is integrable. 
\end{definition}
\begin{remark}
There are several equivalent conditions for integrability:

(a) There exist holomorphic coordinates compatible with $J$. 

(b) The Nijenhuis tensor $N^i_{jk}$ of $J$ vanishes.

(c) The space of $(1,0)$ vector fields relative to $J$ is closed under
Lie bracket.

\noindent We shall assume the reader is familiar with these conditions, see
\cite{Chern}.
\end{remark}

We are interested in $(\mathbb{R}^4, g_E)$, where $g_E$ is the
Euclidean metric. In this case, an OCS is simply a map
\begin{align}\label{OCSmap}
J : \rr^4 \rightarrow \{ M \in SO(4)\>:\> M^2 = - I \}
\end{align}
which satisfies any of the equivalent conditions (a), (b), or (c). 

The first part of the paper deals with a uniqueness question;
the following theorem can be viewed as a ``Liouville Theorem''.
Let $\mathcal{H}^k(\Lambda)$ denote the $k$-Hausdorff measure of a subset $
\Lambda \subset \rr^4$.  The class of $k$-times continuously differentiable
functions is denoted by $C^k$.

\begin{theorem}
\label{t1}
Let $J$ be an orthogonal complex structure of class $C^1$ on
$\mathbb{R}^4 \setminus \Lambda$, where $\Lambda$ is a closed set with
$\mathcal{H}^1(\Lambda) = 0$.  Then either $J$ is constant or $J$ can be
maximally extended to the complement $\mathbb{R}^4 \setminus \{p\}$ of
a point. In both cases, $J$ is the image of the standard
orthogonal complex structure $J_0$ on $\mathbb{R}^4$ under a 
conformal transformation.
\end{theorem}
\begin{remark}
The proof of Theorem \ref{t1} is fairly elementary, and 
is a generalization of \cite[Proposition 6.6]{WoodIJM}, which 
considered the case in which $\Lambda$ is a single point. 
\end{remark}
\begin{remark} This theorem is somewhat reminiscent of the well-known 
  Liouville Theorem in conformal geometry on $\rr^n$ due to
  Caffarelli-Gidas-Spruck \cite{CGS}, which generalized earlier work
  of Obata on $S^n$ \cite{Obata}. In the former, it was proved that a
  positive constant scalar curvature metric on $\rr^n$ conformal to
  the Euclidean metric must be the image of the standard metric on
  $S^n$ under a conformal transformation. 
\end{remark}

Theorem \ref{t1} says that if the Hausdorff dimension 
of the singular set is less than one, then it must correspond to 
a constant OCS. This is sharp, in the sense that there 
exists OCSes with singular set having dimension one. 
Our second uniqueness theorem is as follows: 
\begin{theorem}
\label{t2}
Let $J$ be an orthogonal complex structure 
of class $C^1$ on $\Omega = \rr^4 \setminus \Lambda$,
where $\Lambda$ is a round circle or a straight line, and assume that $J$ is
not conformally equivalent to a constant orthogonal 
complex structure.  Then $J$ is unique up
to sign, and $\Omega$ is a maximal domain of definition for $J$.
\end{theorem}

\begin{remark}
  All such $J$ are conformally equivalent (up to sign) by a conformal 
  transformation identifying the singular circles, 
  and are induced from 
  a ``real'' quadric in $\mathbb{CP}^3$. This correspondence will be made
  explicit in Section \ref{dosrq}. These examples arise in Pontecorvo's
  classification of locally conformally flat Hermitian surfaces, 
  and correspond geometrically to locally conformally
  flat Hermitian metrics on $\cp^1 \times S_g$, where $S_g$ is a Riemann surface
  of genus $g\ge2$ \cite{Pontecorvo}.  
  In Section \ref{dosrq}, using Theorems \ref{t1}
  and \ref{t2}, we will give a new proof of 
  Pontecorvo's classification.
\end{remark}

\begin{remark} We remark that theorems of Bernard Shiffman \cite{Shiffman}
and Errett Bishop \cite{Bishop} are crucial for the proofs of
Theorems \ref{t1} and \ref{t2}, respectively. These theorems give
conditions for when the closure of an analytic set 
is an analytic set, and are generalizations of the 
classical theorem of Thullen-Remmert-Stein, see Section \ref{remove}.
\end{remark} 

  Let $g_{rnd}$ denote a ``round'' metric on the sphere with positive
constant curvature. Since $(S^4 \setminus \{p\}, g_{rnd})$ is
conformally equivalent to $(\mathbb{R}^4, g_E)$, all of our theorems
can of course be phrased on $S^4$.  
The twistor space of $(S^4,g_{rnd})$ is $\mathbb{CP}^3$ -- this is the
total space of a bundle parametrizing orthogonal almost complex
strctures on $S^4$ -- and we let $\pi : \mathbb{CP}^3 \rightarrow S^4$
denote the twistor projection with fiber $SO(4)/U(2) = \mathbb{CP}^1$.
We will see in Section \ref{corresp} that complex hypersurfaces in
$\mathbb{CP}^3$ yield OCSes on subdomains of $S^4$, wherever such a
hypersurface is a single-valued graph.  Conversely, any OCS $J$ on a
domain $\Omega$ yields a holomorphic hypersurface -- the graph of $J$
in the twistor space.  Therefore the problem of finding OCSes on
subdomains is directly related to the geometry of hypersurfaces of
$\mathbb{CP}^3$ under the twistor projection.  It is well known that,
for purely topological reasons, $S^4$ does not admit any global almost
complex structure \cite[page 217]{Steenrod};
this is of course why we
must restrict our attention to proper subdomains of $S^4$.

In Section \ref{class2}, we give a classification of quadric
hypersurfaces under the action of the conformal group, and give a
canonical form for any quadric.  For reasons that will be clear later,
we let $(\xi_0, \xi_{12}, W_1, W_2)$ denote coordinates in
$\mathbb{C}^4$.  The group $SO_{\circ}(1,5)$ of conformal
transformations of $S^4$ acts on $\mathbb{CP}^3$, and is a subgroup of
the holomorphic automorphism group $PGL(4, \mathbb{C})$ of
$\mathbb{CP}^3$, see Section~\ref{action}.

\begin{theorem} 
\label{cf1}
  Any non-singular quadric hypersurface in $\mathbb{CP}^3$ 
is equivalent under the action of $SO_{\circ}(1,5)$ to 
the zero set of 
\begin{align*}
e^{\lambda+i\nu}\xi_0{}\!^2+
e^{-\lambda+i\nu}\xi_{12}{}\!^2+
e^{\mu-i\nu}W_1{}\!^2+
e^{-\mu-i\nu}W_2{}\!^2,
\end{align*}
where $\lambda, \mu, \nu \in \rr$, 
or to the zero set of  
\begin{align}
i(\xi_0^2+\xi_{12}^2) -k \xi_0W_1 + k \xi_{12}W_2 
-\xi_0W_2+\xi_{12}W_1,
\end{align}
where $k \in [0,1).$
\end{theorem}

\begin{remark}
For the complete classification, there are some additional relations
required on $(\l, \m, \n)$ in the diagonalizable case, 
see Lemma \ref{restr} below, 
but $k$ is a complete invariant in the second case. 
To classify the entire moduli space of quadrics, one needs also 
to consider singular quadrics; this 
will appear in a forthcoming paper. 
\end{remark}

Using this we can completely describe the geometry of
non-degenerate quadrics under the twistor projection $\pi$. The
discriminant locus of a quadric $\Q$ is the subset of points $p\in S^4$
for which $\pi^{-1}(p)\cap \Q$ has cardinality different from 2, see Section~\ref{grq}.

\begin{theorem}
\label{quadricst} Let $\Q$ be a non-singular quadric hypersurface in 
$\cp^3$. There are three possibilities:

(0) $\Q$ is a real quadric with discriminant locus a circle in $S^4$,
and $\Q$ contains all of the twistor lines over the circle. 

(1) $\Q$ does not contain any twistor
lines. In this case, the discriminant locus is a torus $T^2 \subset
S^4$ with an smooth unknotted embedding.

(2) $\Q$ contains exactly one or exactly two twistor lines. 
The discriminant locus is a singular torus pinched at one or 
two points, respectively.

\end{theorem}
\begin{remark}
  In Case (1), ``unknotted'' means that the discriminant locus is isotopic
  to a standard torus $T^2 \subset \rr^3 \subset \rr^4 \subset S^4$.
  This is equivalent to saying that $T^2$ bounds an embedded solid
  torus $S^1 \times D^2 \subset S^4$. 
\end{remark}
As discussed above, the quadrics in Case (0) are all conformally equivalent, 
and yield well-defined OCSes
on the complement of a circle or line in $\mathbb{R}^4$.  In Case (1) we have
\begin{theorem}
\label{torusq}
Any quadric containing no twistor lines yields two distinct well-defined 
orthogonal complex structures $J_1$
and $J_2$ on $S^4 \setminus C$, where $C = S^1 \times D^2$ is a solid
torus.
\end{theorem}

We mention that the pair of OCSes on an open set of $S^4$ induced by a
quadric determines a bi-hermitian structure in the sense of
\cite{PontecorvoBi}. See \cite{GatesHullRocek, Kobak, AGG, Hitchin3}
for interesting generalizations of this concept to situations in which
there is no integrable twistor space.
There is also a close relation of the work in this paper 
to the notion of harmonic morphism; related work and examples 
may be found in \cite{Baird}, \cite{BairdWood}, \cite{BairdWood2}, 
\cite{GudWood}, \cite{WoodIJM}.

The methods of this paper involve twistor theory, complex analytic
geometry, and representation theory. Our study of
quadrics began with the realization of a few specific
examples, but a considerable amount of effort was required to formulate the
appropriate canonical forms (Theorem \ref{cf1}) and then to establish the 
behavior arising from generic orbits of the conformal group. The proofs of
Theorems~\ref{quadricst} and \ref{torusq} are completed in
Sections~\ref{smoothness} and \ref{topology}, by combining the above
techniques with some special direct calculations and topological arguments
involving branched coverings.\smallbreak

In a forthcoming paper we will give some related results for the
higher-dimensional case of $\mathbb{R}^{2n}$.

\subsection{Acknowledgements} The authors would like to thank 
Vestislav Apostolov, Paul
Gauduchon, Nigel Hitchin, Claude LeBrun and Max Pontecorvo for
insightful discussions on complex structures. We also thank Denis
Auroux and John Morgan for very helpful conversations regarding the
topology of branched coverings.  The authors worked together at both
the Centro De Giorgi in Pisa and the Institute for Mathematical
Sciences in London.

\section{Twistor geometry}
 
In this section, we discuss some elementary $4$-dimensional twistor
geometry. This material is known (see \cite{AHS, Atiyah, Pontecorvo},
\cite[Section 6]{WoodIJM});
we must include it here for notation and to set our conventions, but
we omit proofs of the key results. 
\subsection{Tangent vectors and differentials}

We consider $\mathbb{R}^{4}$ and take coordinates 
$x_1, y_1, x_2, y_2$.  Letting $z_j = x_j + i y_j$ and $\bar{z}_j = x_j - i y_j$,
define complex one-forms
\begin{align*}
\dzj &= dx_j + i dy_j, \ \ \ \dbzj = dx_j - i dz_j,
\end{align*}
and tangent vectors
\begin{align*}
\pzj &= (1/2) \left(\pxj - i \pyj \right), \ \ \ 
\pzbj = (1/2) \left(\pxj + i \pyj \right).
\end{align*}
The standard complex structure $J_0:T\R^{4} \rightarrow T\R^{4}$ 
on $\mathbb{R}^{4}$ is given by 
\begin{align*}
J_0 (\pp x_j) &=  \pp y_j, \ \ \ J_0 (\pp y_j) = - \pp x_j. 
\end{align*}
Complexify the tangent space, and let 
$T^{(1,0)}(J_0) = \mbox{span}\{\pp z_1, \pp z_2 \}$
be the $i$-eigenspace and 
$T^{(0,1)}(J_0)= \mbox{span}\{\pp \bar{z}_1, \pp \bar{z}_2 \}$ 
the $-i$-eigenspace of $J_0$. 
The map $J_0$ also induces an endomorphism of $1$-forms by
\begin{align*}
 J_0 ( \omega) ( v_1) = \omega (J_0\tp\kern1pt v_1) = - \omega(J_0 v_1),
 \end{align*}
which satisfies
\begin{align*} 
J_0 ( dx_j ) = dy_j, \ \ \ J_0 ( dy_j ) = - dx_j.
\end{align*}
Then 
$\Lambda^{1,0}(J_0) = \mbox{span}\{ dz_1, dz_2 \} $ is the $-i$-eigenspace and 
$\Lambda^{0,1}(J_0) = \mbox{span}\{ d\bar{z}_1, d\bar{z}_2 \} $ 
is the $+i$-eigenspace of $J_0$.

\subsection{Quaternionic vectors and matrices}\label{quat}
A quaternion $q = z_1 + j z_2$ is determined by a pair of complex
numbers $(z_1, z_2)\in\cc^2$. In the early sections of this paper, we
shall identify $\cc^4 = \cc^2 \oplus \cc^2 = \hh \oplus \hh$ by
setting \[ (c_1, c_2, c_3, c_4) = ( c_1 + j c_2,\, c_3 + j c_4).
\] Moreover, we shall view $ \hh \oplus \hh = \hh^2$ as a
{\em{right}} $\hh$-module, and use $j_r : \hh^2 \rightarrow \hh^2$ to
denote right multiplication by $j$:
\begin{align*}
j_r(c_1, c_2, c_3, c_4) = ( - \bc_2, \bc_1, - \bc_4, \bc_3). 
\end{align*}
The composition $\ol{j_r}$ of $j_r$ followed by complex
conjugation is represented by the matrix
\be{JJr}
\mathbb{J}_r = \mat{\K & 0\\0 & \K}, 
\ee
where
\be{JorK} \K = \mat{0&\!-1\\1&0}.\ee 
(This $2\times2$ matrix is sometimes denoted by
$J$ since it represents a standard almost complex structure on $\R^2$,
though that notation would be confusing in our context.)  Thus, right
multiplication by the quaternion $j$ on a column vector is the mapping
$\vv\mapsto\mathbb{J}_r\ol\vv$. For example, the first column of
\rf{JJr} is the transpose of $j_r(1,0,0,0)$.

We may now define $GL(2, \hh)$ as the subgroup of $GL(4,\cc)$
consisting of those $4\times4$ matrices $G$ for which
\begin{align}
\label{GJr}
G\mathbb{J}_r =  \mathbb{J}_r\ol{G},
\end{align}
and $\det G\ne0$.
In these terms, $GL(2, \hh)$ consists of matrices of the form 
\begin{align}
\label{GABCD}
G=\left(
\begin{matrix} 
A & B \\
C & D \\
\end{matrix} \right),
\end{align}
where each $2 \times 2$ submatrix has the form 
\begin{align*}
\left(
\begin{matrix} 
c_1 & c_2 \\
- \bc_2 & \bc_1 \\
\end{matrix} \right).
\end{align*}
It turns out that $\det G$ is real and positive, and $SL(2,\hh)$ is defined to be
the 15-dimensional subgroup of $GL(2,\hh)$ defined by the condition
$\det G=1$. It is in fact isomorphic to the connected group 
$\mathit{Spin}_{\circ}(1,5)$.

\subsection{Twistor space of $S^4$} The twistor projection
$\pi: \mathbb{CP}^3 \rightarrow \mathbb{HP}^1$ is given by
\begin{align}
\label{piz}
\pi([z_1,z_2,z_3,z_4]) = [ z_1 + j z_2, z_3 + j z_4],
\end{align}
where we use right multplication to define quaternionic projective
space, building on the notation of Section~\ref{quat}. Since
$\mathbb{HP}^1$ is naturally isometric to $S^4$ with the round metric,
we can regard the twistor projection as a mapping
$\pi : \mathbb{CP}^3 \rightarrow S^4$.

For reasons that shall soon become apparent, we shall now denote a
point in $\cp^3$ as $[ \xi_0, \xi_{12}, W_1, W_2]$,
and consider a fixed quaternion $q = z_1 + j z_2$. With this new
notation, the fiber $\cp^1=\pi^{-1}([1,q])$ is given by
\begin{align}
\label{suggest}
W_1 + j W_2 &= q ( \xi_0 + j \xi_{12}) 
= \xi_0 z_1 - \xi_{12} \bz_2 + j ( \xi_0 z_2 + \xi_{12} \bz_1). 
\end{align}
This clearly exhibits the fiber as a holomorphic curve in $\cp^3$.

In (\ref{suggest}), $\xi_0$ and $\xi_{12}$ cannot both vanish, and the fiber
is parametrized by $[\xi_0,\xi_{12}]\in\cp^1$. The
fiber over infinity is given by points of the form
$[0,0, W_1, W_2]$,
with $[W_1, W_2] \in \cp^1$. 

\subsection{The twistor fiber}

There are two standard models of the twistor fiber, namely  
\begin{align*}
Z_2^+ & = \{ \mbox{maximal oriented isotropic complex planes in } \cc^4 \},\\
\mathcal{J}_2^+ &=  \{ J \in SO(4) \>:\> J^2 = -I \}.
\end{align*}
To define an isomorphism between these models, 
let $\xi_{0}, \xi_{12} \in \mathbb{C}$,
and define $1$-forms 
\begin{align*} 
\eta_1 = \xi_0 dz_1 - \xi_{12} d\bar{z}_2, \ \ \ \ 
\eta_2 = \xi_0 dz_2 + \xi_{12} d\bar{z}_1.
\end{align*}
If $\xi_0 \neq 0$ then the span of the $\eta_i$ 
defines a maximal isotropic subspace in $\mathbb{R}^{4} \otimes \mathbb{C}
= \mathbb{C}^{4}$, where 
we think of $\mathbb{C}^{4}$ as the space of complex $1$-forms
$\Lambda^1_{\mathbb{C}} = \Lambda^{(1,0)} \oplus \Lambda^{(0,1)}$. 

Define $\psi: \mathbb{CP}^1 \rightarrow Z_2^+$
by $\psi([ \xi_0, \xi_{12}]) = \mbox{span}_{\mathbb{C}} \{ \eta_1, \eta_2\}$.
\begin{proposition}
The map $\psi$ is a biholomorphism, where $Z_2^+$ has the 
complex structure as a Hermitian symmetric space.
\end{proposition}

We will henceforth use this isomorphism to identify $Z_2^+$
with $\cp^1$.
Next, let $J \in \mathcal{J}_2^+$, then $J$ determines a space of
complex $(1,0)$ forms, $\Lambda^{1,0}(J)$.  In general, an
endomorphism $J$ is an orthogonal map (with respect to the Euclidean
metric) if and only if $\Lambda^{1,0}(J)$ is an isotropic subspace
with respect to the complexified Euclidean inner product. This gives
an isomorphism from $\mathcal{J}_2^+$ to $Z_2^+$.  For the inverse
map, simply write a maximal isotropic subspace as a graph over $\cc^2
= \rr^4$, this graph is the corresponding $J$.
Note there is a natural $SO(4)$ action on $\mathcal{J}_2^+$ by
conjugation, and for which the stabilizer of any point is $U(2)$, so
$\mathcal{J}_2^+$ is naturally isomorphic to $SO(4)/U(2)$.
\begin{proposition}
  Under this isomorphism, the matrix $J \in \mathcal{J}_2^+$,
  corresponding to the point $\psi([1,\> \xi\!=\!f\!+\!ig])\in Z_2^+$
  is given by
\begin{align*}
J = \frac{-1}{1 + |\xi|^2} 
\left(
\begin{matrix} 
0  & 1 - |\xi|^2 & 2g & 2f \\
\!-1 + |\xi|^2 & 0 & 2f & -2g \\
-2g & -2f & 0 & 1 - |\xi|^2 \\
-2f & 2g & \!-1 + |\xi|^2 & 0 \\
\end{matrix}
\right).
\end{align*}
\end{proposition}
\begin{remark}
  Note that with our conventions, the standard complex structure $J_0$
  on $\mathbb{R}^{4}$ corresponds to $[1,0]\in\cp^1$, or rather the
  matrix $\mathbb{J}_r$ of (\ref{JJr}).
\end{remark}

\subsection{Holomorphic coordinates}
Next, let $\mathcal{Z}_+(\mathbb{R}^{4})$ denote the twistor bundle of
$\mathbb{R}^{4}$. As a smooth manifold, this is a product
$\mathcal{Z}_+(\rr^{4}) = Z_2^+ \times \rr^{4}$,
but it carries a special complex structure defined as follows:
on the horizontal space (tangent to $\rr^4$), the complex structure is
defined tautologically, while the vertical space (tangent to $\cp^1$)
carries a canonical complex structure.

We next find holomorphic coordinates on $\mathcal{Z}_+(\rr^{4})$.
Motivated from (\ref{suggest}), consider the functions
\begin{align}
\label{Wxi} 
W_1 = \xi_0 z_1 - \xi_{12} \bar{z}_2, \ \ \ \ W_2 = \xi_0 z_2 + \xi_{12} \bar{z}_1.
\end{align}
Define the map 
$\Psi : Z_2^+ \times \rr^{4} \rightarrow \mathbb{CP}^3 \setminus \cp^1$ by
\begin{align}
\label{Psi}
  \Psi\big([\xi_0,\xi_{12}],\,(z_1, z_2)\big) = [\xi_0,\xi_{12},W_1,
  W_2].
\end{align}

\begin{theorem}
The map $\Psi$ is a biholomorphism from 
$ Z_2^+ \times \rr^{4}$ to $\mathbb{CP}^3 \setminus \cp^1$,
where  $Z_2^+ \times \rr^{4}$ has the complex structure 
$J_1$ as the twistor space of $\mathbb{R}^{4}$.
\end{theorem}

The missing fiber $\cp^1_{\infty}$ is given by points with the first two
coordinates equal to zero.
By adding this missing twistor fiber over the point at infinity, and
since $S^4 \setminus \{p\}$ is conformally equivalent to
$\mathbb{R}^4$, we obtain
\begin{corollary}
The map $\Psi$ can be extended to a biholomorphism 
\begin{align} 
\label{biholo}
\hat{\Psi} : \mathcal{Z}_+(S^{4}) \rightarrow \mathbb{CP}^3.  
\end{align}
\end{corollary}
\begin{remark}
It is not hard to verify that under this identification, the 
twistor projection indeed corresponds to the quaternionic 
projection discussed in the beginning of this section.
\end{remark}

\subsection{Complex hypersurfaces and orthogonal complex structures} 
\label{corresp}

We next explain the correspondence between OCSes and holomorphic
submanifolds of $\cp^3$. Background results on this aspect of
twistor theory can be found in \cite{EellsSalamon, Besse,
  deBartolomeisNannicini}.

\begin{theorem}
\label{graphthm}
 Let $\Omega \subset \rr^4$ be a domain.  If $J$ is an
  OCS defined on $\Omega$, then the graph $J(\Omega) \subset \rr^4
  \times \cp^1$ is a holomorphic submanifold.

Conversely, let $H \subset \pi^{-1}(\Omega)$ be a holomorphic 
submanifold such that $H$ intersects each fiber $\cp^1$ in 
exactly one point. Then $H$ is the graph of an OCS. 
\end{theorem}
\begin{proof}
This fact is well known, but we provide a short proof in 
our setting for completeness. 
In each case, we can write the graph or hypersurface 
in $\cp^1 \times \cc^2$ as
\begin{align*}
\{ \phi(z_1, z_2),\,(z_1, z_2) \},
\end{align*}
where $\phi : \Omega \rightarrow \cp^1$. 
Let us assume $\phi(z) \neq [0,1]$ for $z$ in some 
open subset $\Omega_1 \subset \Omega$. 
Then taking an affine coordinate on $\cp^1$, we let 
\begin{align*}
\phi(z_1, z_2) = [1,\,a(z_1, z_2)],
\end{align*}
where $a : \Omega_1 \rightarrow \cc$. 
Consider the subset $\cc \times \Omega_1$ of the total space.  
On this subset, the twistor complex structure has 
$\Lambda^{(1,0)}$ spanned by the forms
\begin{align*}
\{ d a, \  
\omega_1 = dz_1 - a d\bz_2, \ \omega_2 = dz_2 + a d\bz_1 \}. 
\end{align*}
Clearly, $J(\Omega_1)$ is a holomorphic submanifold if and only
if the map $J$ is holomorphic, with the complex structure 
on $\Omega_1$ induced by $J$ itself. 
This is satisfied if and only if
\begin{align*}
da = a_1 dz_1 + a_2 dz_2 + a_{\bar{1}} d\bz_1 + a_{\bar{2}}d \bz_2 
= c_1 ( dz_1 - a d\bz_2) + c_2 ( dz_2 + a d\bz_1). 
\end{align*} 
This clearly implies the equations
\begin{align}
\label{i1}
&a a_1 + a_{\bar{2}} = 0, \ \ \ \  a a_2 - a_{\bar{1}} = 0.
\end{align}
We next write the equations for $J$ to be integrable. 
On the set $\Omega_1$, define
\begin{align*}
v_1 &= a\,\pp z_1 + \pp \bar{z}_2, \ \ \ \ v_2 = a\,\pp z_2 - \pp \bar{z}_1.
\end{align*}
Note that $v_1$ and $v_2$ span $T^{(0,1)}(J)$. 
For $J$ to be integrable, it
suffices to show that $d\omega^1$ and $d\omega^2$ have no $(0,2)$
component relative to $J$, that is,
\begin{align*}
d\omega^1 ( v_1, v_2) &= 0, \ \ \ \ d \omega^2 ( v_1, v_2) =  0. 
\end{align*}
A simple calculation results in equations identical to (\ref{i1}).
The proof is finished by performing an 
analogous argument on the set $\Omega_2 = \{ z \in \Omega : \phi(z) \neq [1 , 0] \}$.
\end{proof}
A corollary of the above proof is 
\begin{corollary} $J$ is integrable if and only if
\begin{align*}
\phi: \Omega \rightarrow \cp^1
\end{align*}
is almost-holomorphic with respect to $J$ itself.
\end{corollary}

\subsection{Action of the conformal group}
\label{action}

We denote by $SO_{\circ}(1,5)$ the identity component of the group
$O(1,5)$, which is the group of orientation 
preserving conformal automorphisms of $S^4$, \cite{Kobayashi}. 
There is an induced action of the group $SO_{\circ}(1,5)$ on
$\mathcal{Z}_+(S^4)$ that we first define abstractly. Let $\phi: S^4 \rightarrow
S^4$ be a conformal automorphism.
If we identify the fiber over a point $p$ with the set of 
orthogonal maps $J: T_p S^4
\rightarrow T_p S^4$ with $J^2 = - I$, then the required action is
\begin{align}
\label{JtoJ}
J_p \mapsto J_{\phi(p)} = (\phi_*)_p\kern1pt J_p\kern1pt (\phi_*)_p^{-1}.
\end{align}
The right hand side is a valid orthogonal map, since $\phi$
satisfies $\phi_* \in CO(4)$. 

We shall now interpret \rf{JtoJ} in terms of the group $SL(2,\HH)$ of
restricted quaternionic transformations defined in
Section~\ref{quat}. A $2\times 2$ quaternionic matrix
\begin{align*}
M = \left(
\begin{matrix}
\tilde{a}  & \tilde{b}  \\
\tilde{c}  & \tilde{d} \\
\end{matrix}
\right), 
\end{align*}
acts on $\hh \oplus \hh$ as follows
\begin{align}
\label{act}
\left(\!\ba{c}q_1\\q_2\ea\!\right)\ \mapsto\
M\left(\!\ba{c}q_1\\q_2\ea\!\right) =
\left(\!\ba{c}\tilde{a} q_1 + \tilde{b} q_2\\
\tilde{c} q_1 + \tilde{d} q_2\ea\!\right).
\end{align}
If we identify $\HH\oplus\HH=\C^4$ and take $M \in SL(2, \hh)$, then
there is an induced action on both $\cp^3$ and $S^4 =
\mathbb{HP}^1$. The latter is given by
\begin{align}
\label{q1q2}
[q_1, q_2] \mapsto [ \tilde{a} q_1 + \tilde{b} q_2,\> 
\tilde{c} q_1 + \tilde{d} q_2],\qquad [q_1,q_2]\in\mathbb{HP}^1.
\end{align}
Using stereographic projection, the corresponding conformal 
map in $\rr^4 = \hh$ is given by 
\begin{align*} 
q  \mapsto (\tilde{c} + \tilde{d}q) (\tilde{a} + \tilde{b} q)^{-1},
\end{align*}
since, by convention, we have used {\em{right}} quaternionic 
multiplication to define $\mathbb{HP}^1$. 
\begin{proposition}
\label{PsiPsi}
Under the identification $\hat\Psi\colon \mathcal{Z}_+(S^4)\to\cp^3$ from
\rf{biholo}, the two actions of $SO_\circ(1,5)$ and $SL(2,\hh)$
coincide.
\end{proposition}
This result is well-known and can be found in
\cite{Atiyah,Pontecorvo}. 
\begin{remark}
For later use, we record the following. 
The inversion $q \rightarrow q^{-1}$
lifts to the automorphism of $\cp^3$ defined by 
\begin{align}
\label{invlift}
[ \xi_0, \xi_{12}, W_1, W_2] \rightarrow [  W_1, W_2, \xi_0, \xi_{12}].
\end{align}
\end{remark}
\section{Degree one solutions and real quadrics}
\label{dosrq}

In this section, we will consider degree one solutions, 
and prove Theorem \ref{t1}.  
We will also introduce the notion of a real 
quadric from the algebraic point of view, examine
its geometry under the twistor projection, 
and prove Theorem \ref{t2}.

\subsection{Removal of Singularities}
\label{remove}
We quote two crucial theorems which will be used later in this
section. The first is due to Bernard Shiffman. 
We refer the reader to \cite[Chapter 2]{EvansGariepy} for background on 
Hausdorff measures. 
\begin{theorem}(\cite{Shiffman}).  Let $U$ be open in $\cc^n$, and let
  $E$ be closed in $U$.  Let $A$ be a pure complex $k$-dimensional analytic
  set in $U \setminus E$, and let $A'$ be the closure of $A$ in $U$.
  If $E$ has Hausdorff $(2k-1)$-measure zero then $A'$ is a pure
  $k$-dimensional analytic set in $U$.
\end{theorem}

The second theorem we will use is the following 
theorem of Errett Bishop.
\begin{theorem}(\cite[Lemma 9]{Bishop})\label{EBish}
Let $U$ be an open subset of $\cc^n$, and $B$ a 
proper analytic subset of $U$.  Let $A$ be an 
analytic subset of $U \setminus B$ of pure complex 
dimension $k$ such that $\ol{A} \cap B$
has $2k$-dimensional Hausdorff measure zero. 
Then $\ol{A} \cap U$ is analytic. 
\end{theorem}
The difference from Shiffman's Theorem is that
the set $B$ is required to be an analytic subset. 
The idea of the proof is the same, by expressing $A$ as a finite 
branched holomorphic covering, and then using a removable singularity 
theorem (the proof of Shiffman is in fact based on Bishop's proof).

 Both of these theorems are generalizations of the classical 
theorem of Thullen-Remmert-Stein which requires the stronger 
assumption that the set $B$ be contained in 
a subvariety of {\em{strictly}} lower dimension \cite{Thullen, RemmertStein}. 

\subsection{Degree one solutions}
A degree one hypersurface in $\mathbb{CP}^3$ is given by a
single linear equation
\begin{align*} 
c_1 \xi_0 + c_2 \xi_{12} + c_3 W_1 + c_4 W_2 = 0,
\end{align*}
where $\xi_0,\xi_{12},W_1,W_2$ are our chosen coordinates and $c_1,
c_2, c_3, c_4$ are constants.

\begin{proposition}
\label{degfib}
 A hyperplane in $\mathbb{CP}^3$ hits every fiber over $S^4$ 
in exactly one point, except over one point in $S^4$, over which 
it contains a twistor line. 
\end{proposition}
\begin{proof}
Since the twistor lines are linear, 
a hyperplane $P$ will intersect a twistor line in exactly one point
or will contain the line. Since $S^4$ is not homeomorphic
to $\cp^2$, $P$ must contain at least one twistor line.  Furthermore, 
it must contain exactly one twistor line because any 
two projective lines in $P$ must intersect.  
\end{proof}
 We next prove Theorem \ref{t1} from the Introduction:
\begin{proof}[Proof (of Theorem \ref{t1})]
Consider the graph of $J$ in $\cp^3$. It is a 
$C^1$ complex hypersurface of $\cp^3 \setminus ( \Lambda \times \cp^1)$.
Such a hypersurface is analytic by a standard 
regularity theorem in several complex variables
\cite{GunningII}.
From our assumption, we have $\mathcal{H}^3( \Lambda \times \cp^1) = 0$.  We
can then use the above theorem of Shiffman \cite{Shiffman}, to
conclude that the closure of the graph of $J$, $\ol{ J( \Omega)}$ is
an analytic subvariety of $\cp^3$.  Note that most fiber $\cp^1$s
transversely intersect $\ol{ J( \Omega)}$ in one point, this implies
$\ol{ J( \Omega)}$ is a degree one subvariety, which must therefore be
linear \cite[page 173]{GriffithsHarris}, \cite[Chapter 5]{Mumford}.  
The conformal equivalence now follows from the fact 
that $SO(5) \subset SO(1,5)$ acts transitively on the dual projective space
$(\cp^3)^{\ast}$ parametrizing hyperplanes in $\cp^3$.
\end{proof}
 Note that in $\rr^4$, the corresponding solution of the 
system (\ref{i1}) is 
\begin{align*}
a(z_1, z_2) = -\frac{c_1 + c_3 z_1 + c_4 z_2}{c_2 - c_3 \bz_2 + c_4 \bz_1}. 
\end{align*}
The singular point is located where both the numerator and 
denominator simultaneously vanish, which corresponds to 
the twistor line contained in the hyperplane. 
\begin{remark}
In the case that $\Lambda$ is a finite collection of points,
one just needs to use the theorem of 
Thullen-Remmert-Stein. This was in fact well-known to experts, 
we thank Paul Gauduchon and Claude LeBrun for informing 
us of this fact. As mentioned in the Introduction, 
this was also previously observed in \cite{WoodIJM}. 
\end{remark}

\subsection{Real quadrics}
\label{algebra}
For notational purposes, it is convenient to 
re-order the coordinates of $\C^4$ by means of the change of basis
\begin{align}
\label{cbasis}
[ \xi_0, \xi_{12}, W_1, W_2] \rightarrow [ \xi_0, W_1, \xi_{12}, W_2]. 
\end{align}
Letting $P$ denote the change of basis matrix, 
we define 
\be{AJA} \mathbb{J} = P\,\mathbb{J}_r P.
\ee
This will lead to an alternative description of $GL(2,\HH)$ that will
be developed in Sections~\ref{class1} and \ref{class2}.

A complex symmetric $4\times4$ matrix $Q$ represents a quadratic form
$q\colon\C^4\to\C$ by means of the formula \be{qf} q(\vv)=\vv\tp Q\vv,
\ee where $\vv$ denotes a column vector. The associated symmetric
bilinear form (SBF) can be recovered from the so-called polarization
formula
\begin{align}
\label{polar}
\textstyle \vv\tp Q\ww  
= \frac12\big( q(\vv + \ww) - q(\vv) - q(\ww)\big).
\end{align}
Henceforth we shall use the terms ``complex symmetric matrix'',
``quadratic form'' and ``SBF'' interchangeably, via \rf{qf} and
\rf{polar}.

In order to analyze such forms in the context of the twistor
projection $\pi\colon\mathbb{CP}^3\to S^4$, we shall define the
quaternionic structure on $S^4$ by means of the $4\times4$ matrix
$\mathbb{J}$ defined in \rf{AJA}, in place of $\mathbb{J}_r$. In other
words, we identify $\C^4$ with $\HH^2$ in the following way.  Set
\[ I = \mat{1&0\\0&1},\] so that \[\mathbb{J} = \mat{0&-I\\I&0}.\] 
Multiplication by the quaternion $j$ now
corresponds to the mapping $\vv\mapsto\mathbb{J}\ol\vv$.

With this new convention, $\mathfrak{gl}(2,\HH)$ consists of those
linear transformations $G$ of $\C^4$ that commute with this mapping,
which in matrix terms means that \be{glH} G\mathbb{J}=\mathbb{J}\ol
G,\ee compare \rf{GJr}. Equivalently, we can state

\begin{definition}
  The space $\mathfrak{gl}(2,\HH)$ consists of matrices $G=\block AB$,
  using the notation \be{AB} \block AB=\mat{A&B\\\!-\ol B&\ol A}\ee
  for this type of block matrix.
\end{definition}
\noindent
In particular, $\mathbb{J}=\block0{\!-\kern-4pt I}$ is itself in
$\mathfrak{gl}(2,\HH)$.

 Since our complex vector space $\C^4$ is endowed with a
\emph{quaternionic} structure, whether or not the entries of a
$4\times4$ matrix are real numbers is of no great concern to us.
The relevant notion of reality is instead the following 
 
\begin{definition}\label{realq}
A quadratic form $q\colon\C^4\to\C$ is \emph{real} if
\begin{align}
\label{real}
q(\mathbb{J}\ol\vv) =\ol{q(\vv)},
\end{align}
whereas $q$ is \emph{purely imaginary} if
\begin{align*}
q(\mathbb{J}\ol\vv) = -\ol{q(\vv)}.
\end{align*}
\end{definition}

We can convert this definition into equations for the associated
symmetric matrix $Q$. Using \rf{polar}, the first condition becomes
\[\ol\vv\tp\mathbb{J}\tp Q\mathbb{J}\ww = \ol{\vv\tp Q\ww},\]
for all $\vv,\ww$. This is equivalent to \be{JQJ} \mathbb{J}\tp
Q\mathbb{J}=\ol Q, \ee which is the same as \rf{glH} (with $Q$ in
place of $G$).  Hence, the symmetric matrix $Q$ represents a real
quadratic form (in the sense of the definition above) if and only if
\be{QAB} Q=\block AB\in\mathfrak{gl}(2,\HH).\ee The fact that $Q$ is
symmetric implies that $(\ol{Q\mathbb{J}})\tp+Q\mathbb{J}=0$, and so
\be{Qt} \tilde Q=Q\mathbb{J}=\block B{\!-\kern-4pt
  A}\in\mathfrak{sp}(2),\ee where we identify the Lie algebra
$\mathfrak{sp}(2)$ with $\mathfrak{gl}(2,\HH)\cap\mathfrak{u}(4)$. In
this case, the $2\times2$ matrix $A$ is symmetric, while
$B\in\mathfrak{u}(2)$.

\begin{definition} 
\label{rqdef}
The real 10-dimensional space of SBFs
  characterized by the equivalent conditions \rf{QAB},\,\rf{Qt} will
  be denoted henceforth by $\Sigma$.
\end{definition}

An obvious element of $\Sigma$ is represented by the identity matrix
\begin{align}
\label{II}
\mathbb{I}=\block I0.
\end{align}
Moreover, we have
\begin{lemma} 
\label{realim}
A SBF $Q$ is real if and only if $iQ$ is purely
  imaginary, and any SBF is represented uniquely by a symmetric
  matrix
\begin{align}
\label{Q12}
Q = Q_1+ iQ_2,
\end{align}
where $Q_1,Q_2\in\Sigma$.
\end{lemma}

\begin{proof} The first statement is an immediate consequence of
  \rf{JQJ}. The second follows from setting \[\textstyle
  Q_1=\frac12\big(Q + \mathbb{J}\tp \ol Q\mathbb{J}\big),\qquad iQ_2
  =\frac12\big(Q - \mathbb{J}\tp\ol Q\mathbb{J}\big).\] Note that
  $Q_2$ (resepctively, $Q_1$) vanishes if and only if $Q$ is real
  (resepctively, purely imaginary).
\end{proof}

Recall that $GL(2,\HH)$ denotes the group of invertible matrices of
the form \rf{AB}, and $SL(2,\HH)$ those with determinant equal to $1$.

\begin{lemma} 
\label{preserve}
If $Q$ is a real SBF and $G\in GL(2,\HH)$ then $G\tp Q G$ is also
real.
\end{lemma}

\begin{proof} We have
\begin{align*}
\mathbb{J}\tp(G\tp Q G)\mathbb{J}\>
=\>(G\mathbb{J})\tp Q(G\mathbb{J})
= (\mathbb{J}\ol G)\tp Q(\mathbb{J}\ol G)
= \ol G\tp (\mathbb{J}\tp Q\mathbb{J})\ol G)
\>=\>\ol{\vphantom{1^1} G\tp QG},
\end{align*}
as required.
\end{proof}

\subsection{Geometry of real quadrics}
\label{grq}

Let $\Q$ be a quadric hypersurface in $\cp^3$. Then $\Q$ intersects
each twistor fiber in one of three possible cases:
(0)  the entire $\cp^1$, (1) one point; (2) two points.
Let $p\in S^4 = \rr^4 \cup \{ \infty \}$. If Case (0) or (1)
happens over $p$, then we say that $p$ belongs to $D$, the
\emph{discriminant locus}.  We write the latter as a disjoint union $D
= D_0 \cup D_1$, where $p\in D_0$ if and only if Case (0) happens over
$p$.

We shall say that a quadric $\Q\subset\cp^3$ is \emph{real} if it is
represented by a SBF that is real in the sense of
Definition~\ref{realq}. 
The collection of real quadrics $\Q$ is an $\mathbb{RP}^9$, 
with maximal rank quadrics corresponding to an open subset.
In the following theorem we will show that a non-degenerate real
quadric is uniquely determined by its discriminant locus.
\begin{theorem}
\label{s1u} 
If $\Q$ is a real non-degenerate quadric, then the
  discriminant locus $D$ is a geometric circle $S^1\subset S^4$, and
  $D = D_0$. Furthermore, $\Q$ is uniquely determined by this circle.
\end{theorem}
\begin{proof} 
Recalling the notation \rf{JorK} and \rf{AB}, we introduce
the matrix 
\begin{align}
\label{Q00} 
Q_0 = \block 0{\K}.
\end{align}
This is compatible with the convention of
Section~\ref{algebra} and the use of coordinates
$[\xi_0,W_1,\xi_{12},W_2]$ on $\cp^3$ (see \rf{AJA}). The
corresponding quadratic form is therefore
\begin{align*}
 - 2\xi_0 W_2 + 2\xi_{12} W_1.
\end{align*}
Using \rf{Wxi}, the fiber equation can be written
\begin{align*}
\begin{split}
  0 &= -\xi_0 ( \xi_0 z_2 + \xi_{12} \bz_1) +
  \xi_{12} ( \xi_0 z_1 - \xi_{12} \bz_2)
 = -z_2 \xi_0^2 + (z_1 -\bz_1) \xi_0 \xi_{12} - \bz_2\xi_{12}^2.
\end{split}
\end{align*}
Clearly, all coefficients vanish when $z_2 = 0$ and $z_1$ is real,
which is a line in $\rr^4$.  The discriminant is
\begin{align*}
- 4\big(( \Im z_1)^2 + |z_2|^2\big),
\end{align*}
which vanishes exactly along the same line.
The quadric is invariant under the inversion (\ref{invlift}), 
so it also contains the twistor line over infinity,
consequently $D = D_0$ is a circle. 
It will be shown below that all real quadrics are
equivalent under the conformal action (Proposition \ref{realcan}). 
Since the conformal group maps circles to circles, 
this proves the first statement.

We next prove the uniqueness statement.  First, fix an $S^1$. Let
$\Q_1$ and $\Q_2$ be two real non-degenerate quadrics with
discriminant locus this fixed $S^1$.  
The lift of the discriminant locus is $\pi^{-1}(D) = S^1 \times \cp^1$
(since an oriented $\cp^1$-bundle over $S^1$ is trivial), and it
disconnects $\Q_2$ into two components.  So just consider a connected
component $C_2$ of $\Q_2 \setminus(S^1 \times \cp^1)$.  Note that
$C_2$ is a degree one hypersurface.  Now look at the closure of $C_2$
in $\cp^3$, and observe that
\begin{align*}
 C_2  \setminus \Q_1\ \subset\ \cp^3 \setminus \Q_1.
\end{align*}
If $C_2 \setminus \Q_1$ is empty, then there is nothing to prove.
Otherwise $C_2 \cap \Q_1$ is a subvariety of dimension one. So we can
then apply Bishop's Theorem (\cite[Lemma 9]{Bishop} $=$
Theorem~\ref{EBish}) to conclude that the closure of $C_2$ is an
algebraic variety in $\cp^3$. But again a degree one subvariety must
be a linear hyperplane, so then $\Q_2$ is degenerate, a contradiction
unless $\Q_1 = \Q_2$.
\end{proof}
We note that the quadratic formula yields an explicit 
solution for the OCS induced by the above quadric, 
\begin{align}
a(z_1, z_2) = \frac{  i \Im z_1 \pm \sqrt{ -(\Im z_1)^2 - |z_2|^2}}{ \bz_2}.
\end{align}
 Theorem \ref{s1u} implies that this OCS is in fact 
invariant under all conformal transformations 
which fix the singular $S^1$.
As a corollary of the above proof we have Theorem \ref{t2} 
from the Introduction:
\begin{proof}[Proof (of Theorem \ref{t2})]
  Given a circle $S^1 \subset S^4$, let $\Q_1$ be the unique real quadric with
  discriminant locus this $S^1$.  In the above proof, now let
  $C_2$ be the graph of an OCS defined on $\R^4\setminus S^1$. If
  $C_2$ is not equal to a branch of the quadric, then it must be a
  hyperplane, in which case the corresponding OCS is conformally
  equivalent to a constant OCS.  The proof is finished by noting that
  the two branches of a real quadric induce $\pm J$. It is also clear
  that the corresponding OCS cannot be extended {\em{smoothly}} to any
  larger domain.
\end{proof}
\begin{remark} It is easy to see from the above proof that
Theorem \ref{t2} remains valid under the more general assumption 
that $\Lambda$ is a finite union of circles. 
\end{remark}

\subsection{Uniformization}
In this section, we show that Theorems \ref{t1} and \ref{t2}, together with 
some well-known results in conformal and Hermitian geometry, 
give a new proof of Pontecorvo's classification of locally 
conformally flat Hermitian surfaces. As pointed out in \cite{Pontecorvo}, this
was also stated without proof in \cite{Boyer1988}. 

\begin{theorem}[Pontecorvo \cite{Pontecorvo}]
If $(M,g,J)$ is a compact connected locally conformally flat Hermitian 
surface, then it is conformally equivalent to one of the following:

1) A complex torus or hyperelliptic surface with the flat metric. 

2) A Hopf surface (finitely covered by $\cc^2 \setminus \{ 0 \} = S^1 \times S^3$)
with the Hermitian metric of Vaisman \cite{Vaisman}.

3) A product $\cp^1 \times S_g$ where $S_g$ is a Riemann surface 
of genus $g \geq 2$ with metric the product of
the $+1$ curvature metric on $\cp^1$ and the $-1$ curvature
metric on $S_g$. 
\end{theorem}
\begin{proof}
A result of Gauduchon states that, under the conditions assumed, 
the scalar curvature $R \geq 0$, and $R \equiv 0$ implies  
that $g$ is K\"ahler \cite[Theorem 1]{Gauduchon} (see also \cite{Boyer1986}). 
Since $R \geq 0$, \cite[Theorem 4.5]{SchoenYau} implies that 
the developing map is injective, and has image $\Omega$, which is the 
domain of discontinuity of the Kleinian group $\pi_1(M)$,
with limit set $\Lambda = S^4 \setminus \Omega$.   
By \cite[Lemma 1.1]{SchoenYaulcf},
there are 2 possibilities: $M$ admits a conformal metric
of strictly positive scalar curvature, 
or $M$ admits a metric of identically zero scalar curvature. 
In the case of positive scalar curvature, 
\cite[Theorem 4.7]{SchoenYaulcf} and \cite[Corollary 3.4]{Nayatani}
imply that $\dim_{\mathcal{H}}(\Lambda) < 1$.
By Theorem \ref{t1}, the complex structure extends to the standard OCS
on $S^4 \setminus \{ p \}$. This implies that $\pi_1(M)$ is 
a subgroup generated by $U(2)$, dilations, and translations. 
The only possible compact quotients are finitely covered 
by a torus, or a Hopf surface. The only ones with positive 
scalar curvature are the latter, in which case 
$\Lambda = \{p_1, p_2\}$, and the metric 
is conformal to the scale-invariant Vaisman metric 
$g = \Vert z \Vert^{-2} g_0$ \cite{Vaisman}.   
In the case $R \equiv 0$, the metric is K\"ahler, 
so $H^2(M) \neq 0$. Using Bourguignon's Weitzenb\"ock formula
\cite[Section 8]{Bourguignon}, Lafontaine showed that in
this case $(M,g)$ is either flat, or a metric in case 3) \cite{Lafontaine}.  
If it is flat, again Theorem \ref{t1} implies
the OCS lifts to the standard OCS on $\rr^4$, which is 
case 1). In case 3), $\Omega = S^4 \setminus S^1$ which is 
conformally equivalent to $S^2 \times H^2$, where $H^2$
is hyperbolic space. Theorem \ref{t2} implies the OCS on
$M$ lifts to a unique OCS on $S^2 \times H^2$, which 
must be the product OCS. Proposition \ref{stabG} below
implies that $\pi_1(M) \subset SU(2) \times SO_{\circ}(1,2)$,
which forces $\pi_1(M) \subset  SO_{\circ}(1,2) = PSL(2, \rr)$. 
\end{proof}

\begin{remark}
If an OCS on a domain $\Omega \subset \rr^4$ is K\"ahler,
then it must be constant. This follows since the K\"ahlerian
condition implies that $J$ is parallel, thus $J$ is constant. 
If one takes the image of a constant OCS under
an inversion, then it is no longer K\"ahler, but it is
locally conformally K\"ahler. As seen in the above proof, 
a real quadric also induces a locally conformally K\"ahler OCS. 
Any other OCS coming from an algebraic hypersurface in $\cp^3$
will not be locally conformally K\"ahler.  This follows from 
the result of Tanno in \cite{Tanno} (see also \cite[Section 3]{Derdzinski}),  
in which it is shown that a locally conformally flat K\"ahlerian space 
must be locally symmetric. 
We point out that the proof of the classification in \cite{Pontecorvo}
relies on some special properties of locally conformally K\"ahler
metrics proved in \cite{PontecorvoTrans}; our proof avoids 
this step. 
\end{remark}

\section{Group actions and stabilizers}
\label{class1}
The main equivalence we will consider is 

\begin{definition}\label{equidef} 
  Two non-zero complex symmetric $4\times4$ matrices $Q,Q'$ are
  \emph{equivalent} if there exists $G\in SL(2,\HH)$ and $\gamma\in
  \C^*$ such that $Q'=\gamma\,G\tp QG$.
\end{definition}

We plan to study the orbits of the complex 10-dimensional vector space
\[ S^2(\C^4)=\Sigma\oplus i\Sigma,\] under the action of the
17-dimensional group \be{17}\C^*\times SL(2,\HH)=U(1)\times
GL(2,\HH).\ee Observe that $-\mathbb{I}$ belongs to $SL(2,\HH)$ (see
\rf{II}), but acts as the identity on $Q$; indeed, $\C^*$ and
$SL(2,\HH)$ only share an identity element.

Since $Q$ and $Q'$ define the same quadric \be{quad} \Q =
\{[\vv]:q(\vv)=0\}\subset \mathbb{CP}^3\ee if and only if $Q'=\gamma
Q$ for some $\gamma\in\C^*$, two quadratic forms are equivalent if and only if
their associated quadrics are related by an element of the connected
group \be{2:1} SL(2,\HH)/\mathbb{Z}_2\cong SO_\circ(1,5),\ee 
as discussed in Section ~\ref{action}. If $Q$ is non-degenerate, we
can use the $\C^*$ action to assume that $\det Q=1$, but it is not
convenient to do this initially.  Indeed, we shall allow $G$ to lie in
$GL(2,\HH)$, the action of which preserves the splitting \rf{Q12} by
Lemma~\ref{preserve}.
We begin the classification theory of quadrics by discussing the case
in which $Q=Q_1$ is itself real in the sense of
Definition~\ref{realq}. We also suppose that the associated quadric
is non-degenerate, meaning that $\mathrm{rank}\,Q=4$.

\begin{proposition} 
\label{realcan}
  Suppose that $Q\in\Sigma$ has rank 4. Then there exists $G\in
  GL(2,\HH)$ such that $G\tp QG=\mathbb{I}$.
\end{proposition} 

\begin{proof}
  If $G\in Sp(2)$ so that $G^{-1} = \ol G\tp$, then by \rf{glH}, \be2
  (G\tp QG)\mathbb{J} = G\tp\tQ\ol G = (\mathrm{Ad}\,G\tp)\tQ\ee in
  the notation of \rf{Qt}. We should therefore concern ourselves with
  the adjoint action of $Sp(2)$ on its Lie algebra $\mathfrak{sp}(2)$.
  Any adjoint orbit of a compact Lie group must intersect a
  fundamental Weyl chamber in the Lie algebra of a maximal torus
  \cite{Adams}. In the case of $Sp(2)$, we may choose a diagonal
  maximal torus $U(1)^2$, and coordinates $\lambda,\mu$ on its Lie
  algebra $\R^2$. The Weyl group includes the reflections in the
  coordinate axes, so there exists $G\in Sp(2)$ so that
  \[(\mathrm{Ad}\,G\tp)\tQ=\left(\!\ba{cccc}
    i\lambda&0&0&0\\0&i\mu&0&0\\0&0&\!-i\lambda&0\\0&0&0&\!-i\mu\ea\!\right)
  =\left(\!\ba{cccc}
    0&0&i\lambda&0\\0&0&0&i\mu\\i\lambda&0&0&0\\0&i\mu&0&0\ea\!\right)
  \mathbb{J},\] with $\lambda,\mu>0$. Thus \be{0D} G\tp QG
  =\mat{0&D\\D&0}=\block 0D,\ee where
  $D=\mathrm{diag}(i\lambda,\,i\mu)$.

  By postmultiplying $G$ by
  $\mathrm{diag}(\lambda^{-1/2},\mu^{-1/2},\lambda^{-1/2},\mu^{-1/2})$,
  we see that \rf{0D} is valid for $D=iI$ and some $G\in GL(2,\HH)$.
  If we now set $H=\block I{iI}$ then $H\tp H= \block 0{2iI}$, and \[
  2G\tp Q G=H\tp\,\mathbb{I}\kern1pt H.\] Replacing $G$ by
  $\sqrt2\,GH^{-1}$ completes the proof.
\end{proof}

The stabilizer of $\mathbb{I}$ by the action of $SL(2,\HH)$ is the
group \be{SO2H} SO(2,\HH) = O(4,\C)\cap SL(2,\HH),\ee where $O(4,\C)$
is the set of complex orthogonal matrices characterized by the
equation $X\tp X=\mathbb{I}$. The group \rf{SO2H} is isomorphic to the
group $SO^*(4)$ described by Helgason \cite[Ch~X,\,\S6]{Helgason}.
Both $SO(4)$ and $SO^*(4)$ are real forms of
\begin{align}
\label{SO4C}
SO(4,\C)\cong SL(2,\C)\times_{\mathbb{Z}_2}SL(2,\C),
\end{align}
and as a counterpart of the well-known isomorphism $SO(4)\cong
SU(2)\times_{\mathbb{Z}_2} SU(2)$, we have $SO^*(4)\cong SL(2,\R)
\times_{\mathbb{Z}_2} SU(2)$. 
We next make explicit its action on the real vector space $\Sigma$.

Let $Q=Q_1$ continue to be a matrix of the form \rf{QAB} representing
a real SBF. 
Then $Q$ decomposes as the
sum of three real symmetric matrices modulo $Q_0$ defined above in 
(\ref{Q00}). 
To see this, given
\rf{QAB}, write $A=L+iM$ and $B=-v\K+iN$ with $v\in\R$ (later $v$ will
be the imaginary part of a complex scalar $\tt$) and $L,M,N$ real.
Then \[ Q = \mat{L+iM& \!-vK+iN\\[4pt] vK+iN & L-iM},\] and the
$2\times2$ matrices $L,M,N$ are all symmetric. In this way, we have
decomposed the tracefree component \be{9+1} Q + vQ_0= \mat{L+iM&
iN\\[4pt] iN & L-iM}=\block{L+iM}{\,iN}\ee of $Q$ (relative to $Q_0$)
into the triple $(L,M,N)$.

The above decomposition of $\Sigma$ into ``$9+1$'' dimensions is
invariant under the stabilizer \be{stab} \mathscr{G} = \{G\in
SL(2,\HH):G\tp Q_0G=Q_0\}\ee that we describe next.

\begin{proposition}\label{stabG}
  The subgroup \rf{stab} equals the set of matrices
  \[\mat{aR &bR\\\!-\ol b R&\ol aR}=\block{aR}{bR},\ \ \hbox{with
  } R\in SL(2,\R),\ a,b\in\C,\ |a|^2\!+\!|b|^2=1.\] It is conjugate
  to $SO(2,\HH)$, and isomorphic to
  $SL(2,\R)\!\times_{\mathbb{Z}_2}\!SU(2)$.
\end{proposition}

\begin{proof} Proposition~\ref{realcan} (and the fact that
  $\det\mathbb{I}= 1 = \det Q_0$) implies that there exists $\FF\in
  SL(2,\HH)$ such that
\begin{align}\label{FQF}
\FF\tp Q_0\FF=\mathbb{I}.
\end{align}
For the sequel, we record one possible choice, namely
\be{F}\FF=\frac1{\sqrt2}\!\left(\!\ba{cccc}
  1&0&0&i\\0&\!-i&1&0\\0&i&1&0\\\!-1&0&0&i\ea\!\right).\ee Let $G\in
SL(2,\HH)$. It follows that
\[G\tp G=\mathbb{I}
\ \Leftrightarrow\ (\FF G)\tp Q_0(\FF G)=\FF\tp Q_0\FF
\ \Leftrightarrow\ \FF G\FF^{-1}\in\mathscr{G},\]
and $\mathscr G$ is certainly conjugate to \rf{SO2H}.

Now suppose that $G=\block{aR}{bR}$ where $R\in SL(2,\R)$ and $|a|^2 +
|b|^2 = 1$, so that \[ U=\left(\!\!\ba{cc}a&b\\-\ol b&\ol
  a\ea\!\right)\in SU(2).\] Recall the definition \rf{Q00} of $Q_0$ in
terms of the $2\times2$ matrix $K$, and observe that $R\tp KR=K$ for
all $R\in SL(2,\R)$. It follows easily that $G\tp Q_0G=Q_0$, and we
obtain have a homomorphism $SL(2,\R)\times SU(2)\to\mathscr{G}$ given
by \be{iso6}(R,U)\mapsto\block{aR}{bR},\ee with kernel
$\{(I,I),(-I,-I)\}$. Its surjectivity follows from the fact that the
stabilizer $\mathscr{G}$ is effectively a subgroup of \rf{SO4C}.
\end{proof}

We now continue the main discussion. Fix any $G\in\mathscr{G}$. Then
the action of $G$ on \rf{9+1} is determined by setting \[ \mat{L'+iM'&
iN'\\ iN' & L'-iM'}= G\tp \mat{L+iM& iN\\ iN & L-iM}G.\] In these
terms, the representation $\mathscr{G}\to\mathrm{Aut}(\R^9)$ is
characterized by the two separate homomorphisms obtained by seeing
what happens when \[ G=\mat{R&0\\0&R}\quad\hbox{ and }\quad
\mat{aI&bI\\\!-\ol bI&\ol a},\] equivalently $U=I$ and $R=I$,
respectively.  The first one is given by \[ (L',\,M',\,N') = (R\tp L
R,\>R\tp M R,\>R\tp N R),\quad R\in SL(2,\R).\] The second requires
more complicated notation, but we can set \[
(\Lambda_{11}',\,\Lambda_{12}',\,\Lambda_{22}') = (U\tp
\Lambda_{11}U,\>U\tp \Lambda_{12}U,\>U\tp \Lambda_{22}U),\quad U\in
SU(2),\] in which (for fixed $\alpha,\beta$) \be{Sab} \Lambda_\ab =
\mat{L_\ab+iM_\ab& iN_\ab\\[4pt] iN_\ab&L_\ab-iM_\ab}\ee is the
complex $2\times2$ matrix whose entries are made up from the $\ab$
entries of $L,M,N$.

Both of these homomorphisms are built up from elementary double
coverings:

\begin{lemma}
  (1) If $L$ is a $2\times2$ real symmetric matrix then $\pi_1(R)(L)=
  R\tp LR$ defines a homomorphism from $SL(2,\R)$ onto the connected
  group $SO_\circ(1,2)$, with kernel $\{I,-I\}$.\\[3pt] (2) If
  $\Lambda$ is a $2\times2$ complex symmetric matrix such that
  $\Lambda \K\in \mathfrak{su}(2)$ then $\pi_2(U)(\Lambda)= U\tp
  \Lambda U$ defines a homomorphism from $SU(2)$ onto $SO(3)$, with
  kernel $\{I,-I\}$.
\end{lemma}

\begin{proof}
(1) The secret is to identify a vector $\vv=(x,y,z)\in\R^3$ with the
symmetric matrix \[ L=L_\vv=\mat{x+y&z\\z&x-y},\] so that
\be{med}\textstyle x=\frac12(L_{11}+L_{22}),\quad
y=\frac12(L_{11}-L_{22}),\quad z=L_{12}.\ee Since $R\tp LR$ is
symmetric, $\pi_1$ is well defined, and its kernel is readily
computed. Since $\det L_\vv$ equals the Lorentzian norm squared of
$\vv$, the image of $SL(2,\R)$ by $\pi_1$ lies inside $SO(1,2)$;
indeed it is a Lie subgroup of maximal dimension and must therefore be
the connected component of the identity.

(2) This is a partially obscure version of the well-known
representation of rotations by unit quaternions.  The secret is to
identify $\ww=(l,m,n)\in\R^3$ with the matrix \be{Kw}
\Lambda=\Lambda_\ww=\left(\ba{cc}l+im& in\\in&l-im\ea\right),\ee so
that $\det \Lambda_\ww=|\ww|^2$. Since
\[\tK=\Lambda \K=\mat{in& -l-im\\l-im &-in}\in\mathfrak{su}(2),\] we
have the correct condition on $\Lambda$. Moreover
\[(U\tp \Lambda U)\K = U\tp\tK\ol U = (\mathrm{Ad}\,U\tp)\tK\in
\mathfrak{su}(2),\] just as in \rf2, and $\pi_2$ is equivalent to the
adjoint representation of $SU(2)$ on $\mathfrak{su}(2)$.
\end{proof}

Now consider the $3\times3$ real matrix \be{X} X=\left(\ba{ccc}
  \frac12(L_{11}+L_{22})&\frac12(M_{11}+M_{22})&
  \frac12(N_{11}+N_{22})\\[5pt]
  \frac12(L_{11}-L_{22})&\frac12(M_{11}-M_{22})&
  \frac12(N_{11}-N_{22})\\[5pt] L_{12} &M_{12} &N_{12}\ea\right).\ee
Its columns encode the matrices $L,M,N$, or more precisely, the
vectors
\[ \vv_1,\qquad\vv_2,\qquad \vv_3\] associated to them via
\rf{med}. This means that if we premultiply $X$ by (the inverse of) an
element of $SO_\circ(1,2)$ then this (right) action coincides with that
of $SL(2,\R)\subset\mathscr{G}$ on \rf{9+1} via $\pi_1$.

As for the rows, these encode the objects \[\ba{ccc}
\frac12(\Lambda_{11}+\Lambda_{22})&\leftrightarrow&\frac12(\ww_1+\ww_2)\\[3pt]
\frac12(\Lambda_{11}-\Lambda_{22})&\leftrightarrow&
\frac12(\ww_1-\ww_2)\\[3pt] \Lambda_{12}&\leftrightarrow&\ww_3,\ea\]
since \rf{Sab} has the same form as the matrix \rf{Kw}. This enables
us to identify the rows with linear combinations of elements
$\ww_1,\ww_2,\ww_3$ in the fundamental representation of $SO(3)$. This
means that if we postmultiply $X$ by an element of $SO(3)$ then this
(right) action coincides with that of $SU(2)\subset\mathscr{G}$ on
\rf{9+1} via $\pi_2$.

We can summarize our discussion by

\begin{proposition}\label{pi1pi2} 
  The action of $\mathscr{G}$ on \rf{9+1} is given by \[ G\cdot
  X=\pi_1(R)^{-1}X\,\pi_2(U),\] in terms of the isomorphism \rf{iso6}
  and identification \rf{X}.
\end{proposition}

This proposition is effectively saying that the real $3\times3$ matrix
$X$, or equivalently the triple $L,M,N$ of $2\times2$ matrices,
encodes an element of the tensor product of the adjoint
representations $\mathfrak{sl}(2,\R)$ and $\mathfrak{su}(2)$. We shall
exploit this next.

\section{Classification of quadrics}
\label{class2}

Consider the complex symmetric matrix \rf{Q12}, whose real part $Q_1$
we assume to have rank 4. Using Proposition~\ref{realcan} and
\rf{FQF}, we may find $H\in GL(2,\HH)$ such that $H\tp Q_1 H = Q_0$.
Using \rf{9+1}, we may write \be{LMN} H\tp QH = Q_0+iH\tp
Q_2H=i\mat{L+iM& iN\\ iN & L-iM}+(1-iv)Q_0.\ee We are at liberty to
act on \rf{X} by elements of $\mathscr{G}$ in an attempt to simplify
$L,M,N$, but $v\in\R$ is invariant by this action.

The \textit{singular value decomposition} (SVD) of an arbitrary real
matrix $X$ of order $m\times n$ asserts that $X$ can be expressed as
$PDQ$ where $P$ is a square orthogonal matrix, $D$ a diagonal matrix
with non-negative entries, and $Q$ a matrix with orthonormal rows
\cite{HornJohnson, Strang}. Hyperbolic versions of this result are
known to hold \cite{Politi}, and we shall now investigate the case
$m=n=3$ in the context of the following definition, whose scope is
limited to this section.

\begin{definition}\label{condia} 
  We shall say that two real $3\times3$ matrices $X,Y$ are
  \emph{congruent} if there exist $P\in SO_\circ(1,2)$ and $O\in
  SO(3)$ such that $PXO=Y$. If $X$ is congruent to a diagonal matrix,
  we shall call it \emph{diagonalizable}.
\end{definition}

We can then prove

\begin{theorem}\label{SVD}
  Any real $3\times3$ matrix $X$ is diagonalizable or congruent to the
  matrix 
\begin{equation}\label{Mk}
M_k=\left(\!\ba{ccc}1&0&0\\1&0&0\\0&k&0\ea\!\right)
\end{equation}
  for some $k\ge0$.
\end{theorem}

This theorem will be established with a series of lemmas that deal
with the cases in which the rank of $X$ is respectively $3,1,2$.

\begin{lemma}\label{SVD3} 
  If $X$ has rank 3 then it is diagonalizable.
\end{lemma}

\begin{proof}
  Suppose that $\rank X=3$. Consider the matrix \[ E =
  \left(\!\ba{ccc}1&0&0\\0&\!-1&0\\0&0&\!-1\ea\!\right)\] representing
  the Lorentzian inner product on $\R^3$. A matrix $P$ belongs to
  $O(1,2)$ if and only if $P\tp EP=E$. The matrix $X\tp E X$ is
  symmetric and represents a quadratic form with the same signature as
  $E$. Thus, there exists $O\in SO(3)$ for which \be{yxu} O\tp(X\tp
  EX)O =\left(\!\ba{ccc}y^2&0&0\\0&-x^2&0\\0&0&-u^2\ea\!\right),\ee
  where $y,x,u$ are all non-zero. 
  
  We now define $P$ by the equation \be{PXO}
  PXO=\left(\!\ba{ccc}y&0&0\\0&x&0\\0&0&u\ea\!\right).\ee Comparing
  \rf{yxu} with \rf{PXO}, we see that \be{simclass} O\tp(X\tp
  EX)O=(PXO)\tp E(PXO),\ee so that \[ (XO)\tp(P\tp EP-E)(XO)=0.\] It
  follows that $P\in O(1,2)$. Finally, we may ensure that $P$ belongs
  to the identity component $SO_\circ(1,2)$ by changing the sign of
  one or both of $y,x$ in \rf{PXO}.
\end{proof}

\begin{lemma}\label{SVD1} If $X$ has rank 1 then it is either
  diagonalizable or congruent to $M_0$.
\end{lemma}

\begin{proof}
  Suppose that $\rank X=1$. Using just the right $SO(3)$ action we
  transform the first row of $X$ into $(a,0,0)$. Since the three rows
  of $X$ are proportional, we see that $X$ is congruent to a matrix
\begin{equation}\label{X00} 
X=\left(\!\ba{ccc} a & 0 & 0\\b & 0& 0\\c&0&0\ea\!\right)
\end{equation}
with two zero columns.
  
  Using the left action by $SO_\circ(1,2)$, we may now convert the
  only non-zero column into one of \be{1col}
  (x,0,0)\tp,\qquad(0,x,0)\tp,\qquad(1,1,0)\tp,\ee according as the
  Lorentzian norm squared $a^2-b^2-c^2$ equals $x^2$, $-x^2$ or $0$.
  To verify the last case, we first make the third coordinate vanish;
  we are then free to change the sign of the middle coordinate, and
  apply a typical element \be{sinh} \mat{\cosh t &\sinh t\\\sinh t
    &\cosh t}\ee of $SO(1,1)$. The matrices associated to the first
  two cases in \rf{1col} are obviously diagonalizable in the sense of
  Definition~\rf{condia}, while the third gives $M_0$.
\end{proof}

The final case is

\begin{lemma}\label{SVD2} If $X$ has rank 2 then it is either
  diagonalizable or congruent to $M_k$ with $k > 0$.
\end{lemma}

\begin{proof}
  Suppose that $\rank X=2$. Using the right $SO(3)$ action we again
  transform the first row of $X$ into $(a,0,0)$. This time, it follows
  that $X$ is congruent to a matrix of the form
  \be{XX} \left(\!\ba{ccc} a & 0 & 0\\b & e&0\\
    c& f &0\ea\!\right),\ee where not both of $e,f$ are zero.  Using
  the action of $SO(2)\subset SO_\circ(1,2)$, we can also assume that
  $e=0$, whence $f\ne0$. The second column is now preserved by
  $SO(1,1)$ acting on the first and second coordinates, and we can use
  this subgroup to convert the first column into one of
  \[(a',0,c)\tp,\qquad (0,b',c)\tp,\qquad (1,\pm1,c)\tp.\]
  In the first subcase, we effectively have a $2\times2$
  block of rank 2, which can be diagonalized by a subgroup
  $SO(1,1)\times SO(2)$ (as in Lemma~\ref{SVD3}). In the second
  subcase, we may diagonalize the result using $SO(2)\times SO(2)$ and
  ordinary SVD. The final subcase yields a matrix congruent to
  \be{11lm} \left(\!\ba{ccc} 1 & 0 & 0\\1 & 0 &
    0\\c&g&0\ea\!\right),\ee where $g=\pm f\ne0$.
  
  Applying $SO(2)\subset SO(3)$, convert \rf{11lm} to
  \[\left(\!\ba{ccc} \cos\theta & \sin\theta & 0\\ 
\cos\theta& \sin\theta & 0\\c\cos\theta-g\sin\theta &
    c\sin\theta+g\cos\theta &0\ea\!\right).\] 
  Choose $\theta=\arctan(c/g)$ so that
  \[ c\cos\theta-g\sin\theta=0,\qquad k=c\sin\theta+g\cos\theta>0.\]
Since $SO_\circ(2,1)$ acts transitively on elements of Lorentzian 
norm squared $-k^2$, we can act by this group so that the second column 
becomes $(0,0,k)\tp$. The new first column remains null and orthogonal to the
  second, and must have the form $(x,x,0)\tp$ with
  $x \neq 0$. We obtain $M_k$ by applying \rf{sinh}.
\end{proof}

Having completed the proof of Theorem~\ref{SVD}, we can interpret the
result in terms of complex $4\times4$ symmetric matrices.

\begin{corollary}\label{cross}
  Let $Q=Q_1+iQ_2$ be a SBF, whose real part $Q_1$ is non-degenerate.
  Then $Q$ lies in the same $GL(2,\HH)$-orbit as one of the matrices
  \be{can1}
  \left(\!\ba{cccc} x+iy &0&0&u+iv-1\\0&-x+iy&u-iv+1&0\\
    0&u-iv+1&-x+iy&0\\u-iv-1&0&0&x+iy\ea\!\right)\ee where
  $x,y,u,v\in\R$, or \be{can2}
  \left(\!\ba{cccc} 2i &-k&0&iv-1\\-k&0&-iv+1&0\\
    0&-iv+1&2i&k\\-iv-1&0&k&0\ea\!\right)\ee where $k,v\in\R$.
  Moreover, the orbit containing \rf{can1} determines $v$, $yxu$,
  $y^2$ and the unordered set $\{x^2,u^2\}$, and the orbit of
  \rf{can2} determines $v$ and $k^2$.
\end{corollary}

\begin{proof}
  Assume that $Q$ already equals the right-hand side of \rf{LMN}.
  Suppose that $X$ is diagonalizable in the sense of
  Definition~\rf{condia}. We can then find $G\in\mathscr{G}$ such
  that, having replaced the left-hand side of \rf{X} by $G\cdot X$,
  the new off-diagonal terms
  \[ L_{11}-L_{22},\quad L_{12},\quad M_{22}+M_{11},\quad M_{12},\quad
  N_{11},\quad N_{22}\] all vanish.  It follows that $L+iM$ is itself
  diagonal and $iN$ is off-diagonal. Equation \rf{PXO} allows us to
  choose
  \be{choice}\ba{rcl} y&=&\>\frac12(L_{11}+L_{22})\>=\>L_{11}\\[3pt]
  x&=&\!-\frac12(M_{11}-M_{22})\>=\>-M_{11}\\[3pt]
  u&=&\!-N_{12}\ea\ee as the diagonal elements of $X$, yielding
  \rf{can1}. If, on the other hand, $X$ is congruent to \rf{Mk} then
  we immediately obtain \rf{can2}.

  For the last statement, note that the canonical forms \rf{can1} and
  \rf{can2} both have real part equal to $Q_0$, and this property will
  only be preserved by a subgroup of $\mathscr G$. The latter leaves
  $v$ invariant, and also the similarity class of $X\tp EX$ (see
  \rf{simclass}). If $X$ is diagonalizable, this class is specified by
  the set $\{y^2,-x^2,-u^2\}$ of eigenvalues in \rf{yxu}, or
  equivalently by the characteristic coefficients
  \be{charpoly}\ba{rcl} p=y^2-x^2-u^2,\quad q=x^2u^2 -u^2y^2 -
  y^2x^2,\quad r=y^2x^2u^2=d^2.\ea\ee Moreover $d=\det X$ is invariant
  by $\mathscr G$. To conclude, note that $p,q,d$ are defined for
  arbitrary $X$, and $M_k$ has $p=-k^2$ and $q=0=d$.\end{proof}

While we now have a complete description of the action of
$GL(2,\HH)$, it remains to describe the effect of the group $U(1)$ in
Definition~\ref{equidef}. This amounts to understanding the action
induced on the canonical forms \rf{can1} and \rf{can2} by multiplying
$Q=Q_1+iQ_2$ by $e^{i\theta}$.

\begin{theorem}
\label{mainclass}
A complex symmetric $4\times4$ matrix $Q$ of rank 4 lies in the same
$U(1)\times GL(2,\HH)$ orbit as either a diagonal matrix \be{st}
Q_{\lambda,\mu,\nu}=
\mathrm{diag}\big(e^{\lambda+i\nu},\,e^{\mu-i\nu},\,
e^{-\lambda+i\nu},\,e^{-\mu-i\nu}\big)\ee with $0\le\lambda\le\mu$ and
$\nu\in[0,\pi/2)$, or a matrix \rf{can2} with $v=0$ and $k\in[0,1)$.
\end{theorem}

\begin{proof} 
  By multiplying $Q$ by a suitable unit complex number, we may suppose
  that the real part $Q_1$ of $Q$ in \rf{realim} also has rank 4, and
  apply Corollary~\ref{cross}.

  First suppose that the $GL(2,\HH)$ orbit of $Q$ contains a matrix
  \rf{can1}. Not only does the matrix $F$ defined by \rf{F} transform
  $Q_0$ into $\mathbb I$, but it transforms \rf{can1} into the
  diagonal matrix $\mathbb{I}+i\block A0$, where \be{tau}
  A=\mat{y-v+i(-x+u)&0\\0&-y-v-i(x+u)}=\mat{\tan\alpha
    &0\\0&\tan\beta}.\ee The second equality can be used here to
  define $\alpha,\beta\in\C$, given that $\det Q\ne0$ and the
  meromorphic function $\tan$ avoids only the values $\pm i$. Choose
  $c,d\in\C$ such that $c^2=\cos\alpha$, $d^2=\cos\beta$.  Acting by
  the matrix \be{cdcd} \mathrm{diag}(c,\>d,\>\ol c,\>\ol d)\in
  GL(2,\HH),\ee we see that $Q$ is equivalent to 
\[\mathrm{diag}\big(e^{i\alpha},
\>e^{i\beta},\>e^{i\ol\alpha},\>e^{i\ol\beta}\big).\] This form makes
the action by $e^{i\theta}\in U(1)$ in \rf{17} transparent, and we may
use it to choose \be{nu-nu} i\alpha=\lambda+i\nu,\qquad
i\beta=\mu-i\nu,\ee so that $\det Q=1$.

  The statement that further restrictions can be imposed on
  $\lambda,\mu,\nu$ follows from Lemma~\ref{restr} below.

  To handle the non-diagonalizable case, denote \rf{can2} by $Q=Q_0+iR_{v,k}$.
  A calculation reveals that $\det Q$ equals
  \be{detQ}   1-2k^2+k^4-6v^2+2k^2v^2+v^4+4iv(k^2+v^2-1)
  =(k^2+(v+i)^2)^2,\ee and we are assuming that this is non-zero.
  The only points $(k,v)$ that are excluded are $(\pm1,0)$, for which
  $\rank Q=3$. The
  real part of the matrix $e^{i\theta}Q$ equals
  \[Q_0\cos\theta-R_{v,k}\sin\theta,\] and the determinant of this
  real part is $\delta^2$ where \be{delta2}\delta = (\cos\theta +
  v\sin\theta)^2+ (k\sin\theta)^2.\ee Since $e^{i\theta}Q$ is again
  non-diagonalizable, Corollary~\ref{cross} implies that \be{HtH}
  H\tp(e^{i\theta}Q)H= Q_0+iR_{\tilde k,\tilde v},\ee for some $H\in
  GL(2,\HH)$. Examining the real part of \rf{HtH} shows that $\det
  H=1/\delta$, whence \be{4ith} e^{4i\theta}\det Q=\delta^2(\tilde
  k^2+(\tilde v+i)^2)^2.\ee Now choose $\theta$ so that the left-hand
  side of \rf{4ith} is real and positive, and (more specifically) take
  $v=\tan\theta$ if $k=0$. Then $\delta\ne0$, and \rf{detQ} tells
  us that $\tilde v=0$.

  Imposing the condition $\det Q=1$ almost fixes a representative of
  the $U(1)$ orbit; the only ambiguity that remains is to multiply $Q$
  by a power of $i$. Setting $v=0 = \tilde{v}$ and $\theta=\pi/2$ in \rf{4ith}
  yields $\tilde k=1/k$, so it suffices to restrict $k$ to the
  interval $[0,1)$.
\end{proof}

Combining the final sentences of the two preceding proofs, we conclude
that the equivalence classes of non-diagonalizable matrices are
faithfully parametrized by $k\in[0,1)$. The situation is
more complicated in the diagonalizable case:

\begin{lemma}\label{restr}
  $Q_{\lambda,\mu,\nu}$, $Q_{\lambda',\mu',\nu'}$ are equivalent in
  the sense of Definition~\ref{equidef} if and only if
  $(\lambda,\mu,\nu)$, $(\lambda',\mu',\nu')$ belong to the same orbit
  under the group $\Gamma$ of transformations of $\R^3$ generated by
  the four maps \be{trans}\left\{\ba{l}
    (\lambda,\,\mu,\,\nu)\mapsto(\lambda,\,\mu,\,\nu+\frac\pi2)\\[2pt]
    (\lambda,\,\mu,\,\nu)\mapsto(-\lambda,\,\mu,\,\nu)\\
    (\lambda,\,\mu,\,\nu)\mapsto(\lambda,-\mu,\,\nu)\\
    (\lambda,\,\mu,\,\nu)\mapsto(\mu,\,\lambda,\,\!-\nu).\ea\right.\ee
\end{lemma}

\begin{proof}
  Suppose that \be{QQ}
  e^{i\theta}H\tp Q_{\lambda,\mu,\nu}H=Q_{\lambda',\mu',\nu'},\ee
  with $H\in GL(2,\HH)$. Taking determinants of both sides shows that
  $e^{4i\theta}=1$, so we only need consider the action by the
  subgroup of $U(1)$ generated by $i$. To obtain the first map in
  \rf{trans}, observe that
\be{map1}
i\,H_1\tp Q_{\lambda,\mu,\nu}H_1=Q_{\lambda,\mu,\nu+\frac\pi2},
\ee
where $H_1=\mathrm{diag}(1,i,1,-i)\in SL(2,\HH)$.

The second map corresponds to $(\alpha,\beta)\mapsto(\ol\alpha,\beta)$
and arises from the equation \be{map2} H_2\tp(Q_{\lambda,\mu,\nu})H_2
= Q_{-\lambda,\mu,\nu},\ee where \[ H_2=\left(\!\ba{cccc}
  0&0&\!-1&0\\0&1&0&0\\1&0&0&0\\0&0&0&1\ea\!\right)\in SL(2,\HH).\]
The simultaneous conjugation $(\alpha,\beta)\mapsto
(\ol\alpha,\ol\beta)$ is realized by \be{map3}
H_3{}\!\tp(Q_{\lambda,\mu,\nu})H_3 = Q_{-\!\lambda,-\!\mu,\,\nu}\ee
with $H_3=\block 0I$. Thus the third map in \rf{trans} is generated by
$H_2H_3$. The final map arises by swapping $\alpha$ and $\beta$, and
the equation \be{map4} H_4{}\!\tp(Q_{\lambda,\mu,\nu})H_4 =
Q_{\mu,\lambda,-\!\nu},\ee with $H_4=\block{\K}0$.

For the `only if' statement, we may suppose that $\nu\in[0,\pi/2)$ and
use \rf{nu-nu} to write \be{Qdiag} Q_{\lambda,\mu,\nu}=\mathrm{diag}
\big(e^{i\alpha},\,e^{i\beta},\,e^{i\ol\alpha},\,e^{i\ol\beta}\big).\ee
Since $\cos\alpha$ and $\cos\beta$ are both non-zero, we can reverse
the procedure \rf{cdcd} to obtain a matrix \rf{can1} in the same
$GL(2,\HH)$ orbit. We now appeal to the last part of
Corollary~\ref{cross} to conclude that $y,x,u$ are uniquely specified
up to changing signs of any two of them or swapping $x,u$. But these
operations exactly correspond to the maps above, by means of \rf{tau}.
Indeed, $x\leftrightarrow u$ translates into \rf{map2},
$(y,x,u)\mapsto(y,-x,-u)$ into \rf{map3}, and $(y,x,u)\mapsto(-y,x,-u)$
into \rf{map4}. The description of $\Gamma$ is completed by the
discussion leading to \rf{map1}.
\end{proof}

In view of the lemma, we may now represent each orbit of $\Gamma$ by a
unique point of the domain \[
\{(\lambda,\mu):0\le\lambda\le\mu\}\times S^1,\] where $S^1$
parametrizes $e^{4i\nu}$, except for the identification of
$(\mu,\mu,\nu)$ with $(\mu,\mu,\frac\pi2\!-\!\nu)$. Examples of
diagonal matrices $Q_{\lambda,\mu,\nu}$ for which the stabilizer of
$(\lambda,\mu,\nu)$ in $\Gamma$ is not the identity are

(a) the real quadratic form $Q_{0,0,0}=\mathbb{I}$;

(b) $Q_{0,0,\nu}$ with $\nu\not\equiv0$ modulo $\pi/2$;

(c) $Q_{0,\mu,\nu}$ with $\mu\ne0$, which is equivalent to
$\mathbb{I}+iQ_2$ where $\rank Q_2=2$.

(d) $Q_{\lambda,\lambda,0}$ and $Q_{\lambda,\lambda,\frac\pi4}$ with 
$\lambda\ne0$.

\noindent Most of these distinguished quadrics will play a role in the
sequel.

\begin{remark}
  The half line of forms $Q_{\lambda,\lambda,0}$ with $\lambda\ge0$ is
  realized by taking $y=u=v=0$ and $x=\tanh\lambda\in[0,1)$; it
  follows that $q=d=0$ and $p=-x^2$ in \rf{charpoly}. The
  non-diagonalizable solutions also have $q=d=0$ and $p=-k^2\le0$.  It
  follows that the invariants $p,q,d$ do not separate the $U(1)\times
  GL(2,\HH)$ orbits, and the associated quotient space is not
  Hausdorff.
\end{remark}

\subsection{Quadrics and twistor lines}
\label{LQ}

As a first application of the classification 
theorem, we identify exactly how many twistor 
lines a quadric may contain.  
We use the notation from Section \ref{grq}.
\begin{theorem}
\label{quadricdisc}
For any non-degenerate quadric $\Q$, there are three possibilities:

(0) $D = D_0 = S^1$, which is the case when $\Q$ is real.

(1) $\Q$ contains no twistor lines, that is, $D = D_1$. %

(2) $\Q$ contains exactly one or exactly two twistor lines. 
\end{theorem}

Since the statement is conformally invariant, we may assume that the
quadratic form $q$ is in the canonical form of
Theorem~\ref{mainclass}, with the coordinates of $\C^4$ taken in the
order $\xi_0,W_1,\xi_{12},W_2$. Theorem \ref{quadricdisc} therefore follows from
\begin{proposition}
\label{allprop}
Let $Q$ be in canonical form (\ref{st}),
with $ 0\le \l \le \m$, and $0 \le \n < \pi/2$. The corresponding quadric
$\Q$ contains exactly two twistor lines if and only if $\lambda = \mu \ne0$,
and $\nu = 0$.  The quadric $\Q$ contains a family of twistor lines over a
circle if and only if $\lambda = \mu = \nu= 0$.  Otherwise, $\Q$
contains no twistor lines.
 
In the case $Q = Q_0 + i R_{k,0}$ of (\ref{can2}) with $k \in [0,1)$,
the corresponding quadric $\Q$ contains exactly one twistor line.

\end{proposition}

\begin{proof} 
  Using our identification \rf{Psi} of $\mathbb{CP}^1 \times
  \mathbb{R}^4$ with $\mathbb{CP}^3 \setminus \mathbb{CP}^1$, $\Q$ is
  defined by
\begin{align*}
e^{ \lambda + i \nu}(\xi_0)^2 +  e^{ - \lambda + i \nu} (\xi_{12})^2  
+  e^{ \mu - i \nu}(  \xi_0z_1 - \xi_{12}\bz_2)^2
+ e^{ - \mu - i \nu}(  \xi_0z_2 + \xi_{12}\bz_1)^2 = 0.
\end{align*}
Write this as 
\begin{align}
\label{ABC}
A (\xi_0)^2 + 2B \xi_0 \xi_{12}  + C (\xi_{12})^2 = 0,
\end{align}
with 
\begin{align}
\label{ABC2} 
\begin{split}
\left\{\ba{rcl}
A &=&  c_1 + c_3 (z_1)^2 + c_4 (z_2)^2,\\[3pt]
B &=&  \!-c_3 z_1 \bz_2 + c_4 \bz_1 z_2,\\[3pt]
C &=& c_2 + c_3 (\bz_2)^2 + c_4 (\bz_1)^2,
\ea\right.
\end{split}
\end{align}
where to simplify notation we have set
\begin{align}
\label{c1c2c3c4}
c_1 = e^{ \lambda + i \nu},\
c_2 =  e^{ - \lambda + i \nu},\
c_3 = e^{ \mu - i \nu},\
c_4 =  e^{ - \mu - i \nu}.
\end{align}
The subset of the discriminant locus in $\rr^4$ 
containing entire twistor fibers is given by
\begin{align*}
D_0 = \{  (z_1, z_2)\>:\> A = 0,\ B = 0,\ C = 0 \}.
\end{align*}
If $(z_1, z_2)\in D_0$, then
\begin{align}
\nonumber
e^{ \lambda + i \nu} + e^{ \mu - i \nu} (z_1)^2 +  e^{ - \mu - i \nu} (z_2)^2 &= 0\\
\nonumber
e^{ - \lambda + i \nu} + e^{ \mu - i \nu} (\bz_2)^2 + e^{ - \mu - i \nu} (\bz_1)^2 &= 0\\
\label{e3}
e^{ \mu - i \nu} z_1 \bz_2 - e^{ - \mu - i \nu}\bz_1 z_2 & = 0.
\end{align}
Clearly $(z_1, z_2) \ne(0,0)$, since the coefficients are non-zero. 
For the same reason, these quadrics do not contain the twistor line over infinity,
as easily follows from applying the inversion (\ref{invlift}).

Let us examine first the case that $z_1 \neq 0$ and $z_2 = 0$. 
Then we have
\begin{align*}
e^{ \lambda + i \nu} + e^{ \mu - i \nu} (z_1)^2 &= 0\\
e^{ - \lambda + i \nu}+ e^{ - \mu - i \nu}(\bz_1)^2 &= 0.
\end{align*}
This yields 
\begin{align*}
(z_1)^2 = - e^{ \l - \m + 2i\n} = - e^{\m - \l - 2i \n},
\end{align*}
which implies that $\l = \m$ and $\n = 0$, 
in which case there are exactly two solutions
$(\pm i, 0)$. 

Similarly, if $z_1 = 0$ and $z_2 \ne0 $, then 
\begin{align*}
(z_2)^2 = - e^{ \l + \m + 2i\n} = - e^{ -\l - \m - 2 i \n},
\end{align*}
which implies that $(\l, \m, \n) = (0,0,0)$ since by assumption
$0 \le \l \le \m$. 

We next consider the case that we have a solution with 
both $z_1$ and $z_2$ non-zero.  
From (\ref{e3}) we obtain
\begin{align}
\label{e3'h}
e^{ \mu } z_1 \bz_2 = e^{ - \mu }\bz_1 z_2.
\end{align}
Taking norms, it follows that $\m = 0$, and therefore 
also $\l =0$.
The equations then simplify to
\begin{align}
\label{e1'}
e^{ 2 i \nu} + (z_1)^2 + (z_2)^2 &= 0\\
\label{e2'}
e^{2 i \nu} + (\bz_2)^2 + (\bz_1)^2 &= 0\\
\nonumber
z_1 \bz_2 - \bz_1 z_2 &= 0.
\end{align}
Equations (\ref{e1'}) and (\ref{e2'}) imply that $\n =0$.
Consequently, $\Q$ is a real quadric, and this case 
was proved in Theorem \ref{s1u}. 

In the non-diagonalizable case, the quadratic form is twice
\begin{align*}
 i(\xi_0^2+\xi_{12}^2) -k \xi_0W_1 + k \xi_{12}W_2 
-\xi_0W_2+\xi_{12}W_1.
\end{align*}
Using (\ref{Psi}), 
the quadric is then written in the form (\ref{ABC}), with
\begin{align}
\label{ndABC}
A=i + k \kern1pt\bz_1- \bz_2,\quad B=\Re(k z_2)+i \Im(z_1),
\quad C=i - k z_1- z_2.
\end{align}
To find twistor fibers, we set $A=B=C=0$.
From $A+\overline C= -2 \bz_2$, we get $z_2=0$. 
From $B=0$,
we get $z_1=x_1\in\R$. From $C = 0$, we obtain
$i -k x_1 = 0$, which clearly has no solution. 
Thus there are no fibers in $\mathcal{Q}$ over $\rr^4$. 

We next look at the point at infinity by performing 
the inversion (\ref{invlift}).  
The expression for the inverted $Q$ is twice
\begin{align}
\label{invnon}
 i(W_1^2+W_2^2) - k W_1 \xi_0 +k W_2 \xi_{12} 
-W_1 \xi_{12} + W_2 \xi_0.
\end{align}
The new $A, B, C$ have no constant 
term, so the corresponding inverted quadric contains 
the fiber over the origin. 
\end{proof}

\section{Smooth points}
\label{smoothness}

Next we examine the smooth points on the discriminant locus.
In this section, we will consider the diagonalizable 
case (\ref{st}), and assume throughout that
$0\le\lambda\le\mu$, and $\nu\in[0,\pi/2)$.
The non-diagonalizable case will be considered 
seperately in Proposition \ref{ndp} below. 
Writing the quadric in the form (\ref{ABC}), 
the discriminant is defined by $\Delta = B^2 - AC$, which is
\begin{align}
\label{ABC3}
\Delta 
& =
(-c_3 z_1 \bz_2 + c_4 \bz_1 z_2)^2 -  (  c_1 + c_3 (z_1)^2 + c_4 (z_2)^2)
( c_2 + c_3 (\bz_2)^2 + c_4 (\bz_1)^2),
\end{align}
and the discriminant locus is given by 
\begin{align}\label{discD}
D = \{(z_1, z_2) \in \rr^4\>:\>\Delta(z_1, z_2) = 0 \}.
\end{align} 
We also adopt the following notation: if $F: \rr^4 \rightarrow
\rr^2=\cc$ with
\begin{align*}
F = 
\left(
\begin{matrix}
f \\
g \\
\end{matrix}
\right)
= f + i g,
\end{align*}
set
\begin{align*}
\begin{split}
d_{\rr} F &= 
\left(
\begin{matrix}
\partial_{x_1} f & \partial_{y_1} f & \partial_{x_2} f & \partial_{y_2}f\\[4pt] 
\partial_{x_1} g & \partial_{y_1} g & \partial_{x_2} g & \partial_{y_2}g\\ 
\end{matrix}
\right)
= \left(
\begin{matrix}
\partial_{x_1} F & \partial_{y_1} F & \partial_{x_2} F & \partial_{y_2}F\\ 
\end{matrix}
\right).
\end{split}
\end{align*}
\begin{proposition}
The matrix of $d_{\rr} F \otimes \cc : \cc^4 \rightarrow \cc^2$ 
in the bases 
\begin{align*}
\left\{
{\partial / \partial z_1},\,
{\partial / \partial z_2},\,
{\partial / \partial \bar{z}_1},\,
{\partial / \partial \bar{z}_2}
\right\},\quad
\left\{
{\partial / \partial w},\,
{\partial / \partial \bar{w}}
\right\}
\end{align*}
is given by 
\begin{align*}
d_{\rr} F \otimes \cc
=
\left(
\begin{matrix}
 \partial_{z_1} F & \partial_{z_2} F &
\partial_{\ol{z}_1} F & \partial_{\ol{z}_2} F \\[4pt] 
 \partial_{z_1} \ol{F} & \partial_{z_2} \ol{F} &
\partial_{\ol{z}_1} \ol{F} & \partial_{\ol{z}_2} \ol{F}
\end{matrix}
\right).
\end{align*}
\end{proposition}
\begin{proof}
This is an elementary computation. 
\end{proof}
It follows from the implicit function theorem 
that the set $F^{-1}(\{0\})$ is a submanifold
provided that $d_{\rr} F (p) : \rr^4 \rightarrow \rr^2$
has rank $2$ for each $p \in F^{-1}(\{0\})$.

\begin{proposition}
  The rank of $d_{\rr}F(p)$ is strictly less than $2$ if and only if
\begin{align*}
(\partial_{\ol{z}_1} F ,\, \partial_{\ol{z}_2} F) 
= e^{i \theta}\,
(  \ol{\partial_{z_1} F} ,\, \ol{\partial_{z_2} F})
\end{align*}
for some $\theta \in[0, 2\pi)$.
\end{proposition}
\begin{proof}
If this condition is satisfied, then 
\begin{align*}
d_{\rr} F \otimes \cc
&=
\left(
\begin{matrix}
 \partial_{z_1} F & \partial_{z_2} F &
 e^{i \theta} \ol{\partial_{z_1} F} &  e^{i \theta}
 \ol{\partial_{z_2} F}  
\\[4pt] 
  e^{-i \theta}  {\partial_{z_1} F}   &    e^{-i \theta} {\partial_{z_2} F} &
\partial_{\ol{z}_1} \ol{F} & \partial_{\ol{z}_2} \ol{F}
\end{matrix}
\right)\\[4pt] 
& = 
\left(
\begin{matrix}
 \partial_{z_1} F & \partial_{z_2} F &
 e^{i \theta} \partial_{\bz_1} \ol{F} &  e^{i \theta} \partial_{\bz_2}
 \ol{F}  
\\[5pt] 
  e^{-i \theta}  {\partial_{z_1} F}   &    e^{-i \theta} {\partial_{z_2} F} &
\partial_{\ol{z}_1} \ol{F} & \partial_{\ol{z}_2} \ol{F}
\end{matrix}
\right)
\end{align*}
The second row is $e^{-i \theta}$ times the first, so the rank is not
maximal.

For the converse, without loss of generality 
assume that the first row is non-zero. 
The rank is strictly less than $2$ if and only if
there exists a constant $c$ so that 
\begin{align*}
\begin{split}
d_{\rr} F &= 
\left(
\begin{matrix}
\partial_{x_1} f & \partial_{y_1} f & \partial_{x_2} f & \partial_{y_2}f\\[4pt]  
c\partial_{x_1} f & c\partial_{y_1} f & c \partial_{x_2} f & c \partial_{y_2}f\\ 
\end{matrix}
\right).
\end{split}
\end{align*}
A computation shows that this corresponds to 
\begin{align*}
e^{ i \theta} = \frac{ (1-c^2) + 2 c i }{ 1 + c^2},
\end{align*}
so that $c$ equals $\tan(\theta/2)$.
\end{proof}

We let $\nabla_{\cc}$ denote the complex gradient of a function, so that
\begin{align*}
\nabla_{\cc} F =(  \partial_{z_1} F,\,\partial_{z_2} F,\,
\partial_{\ol{z}_1} F,\,\partial_{\ol{z}_2}F ).
\end{align*}
We now apply this to the discriminant defining \rf{discD}.
\begin{proposition}
We have 
\begin{align*}
\begin{split}
\nabla_{\cc} 
(\Delta)
= \Big(\! &-\!2 c_3 c_4 \bz_1 |\zz|^2 - 2 c_2 c_3 z_1,\> 
  -2 c_3 c_4 \bz_2 |\zz|^2 - 2 c_2 c_4 z_2, \\
&  -2 c_3 c_4 z_1 |\zz|^2 - 2 c_1 c_4 \bz_1,\> 
  - 2 c_3 c_4 z_2 |\zz|^2 - 2 c_1 c_3 \bz_2\> \Big),
\end{split}
\end{align*}
where $|\zz|^2 = |z_1|^2 + |z_2|^2$. 
\end{proposition}
\begin{proof}
  Referring to \rf{ABC2}, we have the easy formulae
\begin{align*}
\begin{split}
\nabla_{\cc} A &= ( 2 c_3 z_1, 2 c_4 z_2, 0 ,0)\\
\nabla_{\cc} B & = ( - c_3 \bz_2, c_4 \bz_1, c_4 z_2, -c_3 z_1)\\
\nabla_{\cc} C & = ( 0 , 0, 2 c_4 \bz_1, 2 c_3 \bz_2).
\end{split}
\end{align*}
The claim follows applying these to 
$\Delta = B^2 - A C $. 
\end{proof}
We examine the bad points, i.e., points at which there 
exists $\theta \in [0, 2 \pi)$ such that 
\begin{align}
\begin{split}
\label{badp}
 - 2 c_3 c_4 z_1 |\zz|^2 - 2 c_1 c_4 \bz_1 &= e^{i \theta} 
\big(\!-2 \ol{c}_3 \ol{c}_4 z_1 |\zz|^2 - 2 \bc_2 \bc_3 \bz_1\big)\\
 - 2 c_3 c_4 z_2 |\zz|^2 - 2 c_1 c_3 \bz_2 
& =  e^{i \theta}\big(\!-2\bc_3\bc_4 z_2 |\zz|^2 - 2 \bc_2 \bc_4 \bz_2\big).
\end{split}
\end{align}
Note that $(z_1, z_2) = (0,0)$ is trivially a solution, but this point
is clearly not on the discriminant locus.

\smallbreak 
It is convenient to divide the following analysis of potential 
singular points into three cases:
\par \vskip2pt (i) $z_2 = 0$ and $z_1 \ne0$; \quad (ii)
$z_1 = 0$, and $z_2\ne0$; \quad (iii) $z_1 \ne0$ and $z_2 \ne0$. 
\vskip2pt \par \noindent
We shall now consider in turn each of these three
cases and related examples.
\subsection{Case (i) and twistor fibers} Assume that $z_2 = 0$ but
$z_1\ne0$. We now have only the one equation
\begin{align}
\label{badp2}
 \Big| c_3 c_4 z_1 |z_1|^2 + c_1 c_4 \bz_1 \Big| &=
\Big| \ol{c}_3 \ol{c}_4 z_1 |z_1|^2 + \bc_2 \bc_3 \bz_1 \Big|.
\end{align}
We must also have the discriminant vanish:
\begin{align}
\label{discp2}
0 = \Delta &= -(c_1 + c_3 z_1^2)(c_2 + c_4 \bz_1^2)
= - c_1 c_2 - c_2 c_3 z_1^2 - c_1 c_4 \bz_1^2 - c_3 c_4 |z_1|^4.
\end{align}
Recalling (\ref{c1c2c3c4}),
the equations (\ref{badp2}) and (\ref{discp2}) may be rewritten as
\begin{align}
\label{c11}
( e^{ -2 i \n + \l - \m}  -   e^{ 2 i \n - \l + \m}) z_1^2 + 
( e^{ +2 i \n + \l - \m} -  e^{ -2 i \n - \l + \m} ) \bz_1^2 &= 
e^{2 ( \m - \l)} - e^{2( \l - \m)}\\
\label{c12}
e^{2 i \n} + e^{\m - \l} z_1^2 + e^{\l - \m} \bz_1^2
+ e^{-2 i \n} |z_1|^4 &= 0.
\end{align}
Writing out the real and imaginary parts, we obtain three real 
equations
\begin{align}
\label{e1c}
\sinh( \l - \m) \cos (2 \n) (x^2 - y^2) + 2 \cosh(\l - \m) \sin (2 \n) x y
= - \sinh( \l - \m) & \cosh( \l - \m)\\
\label{e2c}
2 \cosh( \l - \m) ( x^2 - y^2) + \cos(2 \n)(1 + (x^2 + y^2)^2) & = 0\\
\label{e3c}
-4 \sinh ( \l - \m) x y + \sin ( 2 \n) (1 - (x^2 + y^2)^2) & = 0.
\end{align}
\vskip3pt

\begin{proposition} 
\label{casei} If Case (i) happens, then 
$\l = \m$ and $\n = 0$. 
\end{proposition}
\begin{proof}
First, let us assume that $\l\ne\m$. 
Equations (\ref{e2c}) and (\ref{e3c}) give 
\begin{align*}
x^2 - y^2 &= - \frac{ \cos ( 2 \n)}{ 2 \cosh( \l - \m)} (1 + |z_1|^4), \ \ \ \ 
xy = \frac{ \sin (2 \n)}{ 4 \sinh (\l - \m)} (1 - |z_1|^4).
\end{align*}
Substituting these into (\ref{e1c}), we have
\begin{align*} 
\begin{split} 
&\sinh( \l - \m)  \cos(2\n)\Big(\!-\frac{ \cos ( 2 \n)}
{ 2 \cosh( \l - \m)} (1 + |z_1|^4)  \Big) \\ & + 2 \cosh(\l - \m) \sin
(2\n)\frac{ \sin (2 \n)}{ 4 \sinh (\l - \m)}(1 - |z_1|^4)
= - \sinh( \l - \m) \cosh( \l - \m)\\
\end{split}
\end{align*}
Multiplying by $2 \sinh \cosh$, 
\begin{align*}
\begin{split} 
-\sinh^2( \l - \m) \cos^2 (2 \n)(1 + |z_1|^4) &+ \cosh^2(\l - \m) 
\sin^2 (2 \n)(1 - |z_1|^4)\\[3pt]
& = -2 \sinh^2( \l - \m) \cosh^2( \l - \m).
\end{split} 
\end{align*}
This is 
\begin{align*}
\begin{split} 
&|z_1|^4 \Big( -\sinh^2( \l - \m) \cos^2 (2 \n) - \cosh^2(\l - \m) \sin^2 (2 \n) \Big)\\
&=  +\sinh^2( \l - \m) \cos^2 (2 \n) - \cosh^2(\l - \m) \sin^2 (2 \n)
-2 \sinh^2( \l - \m) \cosh^2( \l - \m),
\end{split} 
\end{align*}
which simplifies to 
\begin{align*}
\begin{split} 
|z_1|^4 \Big(\!-\sinh^2( \l - \m) & \cos^2 (2 \n) - \cosh^2(\l - \m) \sin^2 (2 \n) \Big)\\
&= \fr12 \Big( \cos( 4 \n) - \cosh\big( 2( \l - \m)\big)\Big)
\cosh \big( 2 ( \l - \m) \big). 
\end{split} 
\end{align*}
A computation shows that then 
\begin{align}
\label{bb}1 + |z_1|^4 &= 2 \cosh^2( \l - \m), \ \ \ \ 
1 - |z_1|^4 = -2 \sinh^2 ( \l - \m).
\end{align}
Substituting these into (\ref{e2c}) and (\ref{e3c}), 
and squaring the resulting equations, we obtain
\begin{align*}
x^4 - 2 x^2 y^2 + y^4 = \cos^2 (2 \n) \cosh^2 ( \l - \m), \ \ \ \ 
x^2 y^2 = \fr14\sin^2 (2 \n) \sinh^2 ( \l - \m).
\end{align*}
From these we compute
\begin{align*}
1 + |z_1|^4 &= 1 + x^4 + 2 x^2 y^2 + y^4
= 1 + x^4 - 2 x^2 y^2 + y^4 + 4 x^2 y^2\\
& = 1 +  \cos^2 (2 \n) \cosh^2 ( \l - \m) + \sin^2 (2 \n) \sinh^2 ( \l - \m). 
\end{align*}
This simplifies to 
\begin{align*}
1 + |z_1|^4 = \fr12 \Big( 2 + \cos( 4 \n) + \cosh\big( 2( \l - \m)\big)
\Big).
\end{align*}
From the above equation (\ref{bb}), the coefficients must 
therefore satisfy
\begin{align*} 
2 \cosh^2( \l - \m) -  \fr12\cosh\big( 2( \l - \m) \big) = 1 +  \fr12\cos( 2 \n).
\end{align*} 
This simplifies to 
\begin{align*} 
1 + \fr12 \cosh \big( 2 ( \l - \m) \big)  =  1 + \fr12\cos( 2 \n).
\end{align*}
We conclude that $ \cosh \big( 2 ( \l - \m) \big) = \cos( 2 \n)$,
which is contrary to $\l \neq \m$. 

 Finally, we consider the case $\l = \m$. 
In this case, the above system is 
\begin{align*}
( e^{ -2 i \n }  -   e^{ 2 i \n }) z_1^2 + 
( e^{ +2 i \n } -  e^{ -2 i \n } ) \bz_1^2 &= 0\\
e^{2 i \n} + z_1^2 + \bz_1^2 + e^{-2 i \n} |z_1|^4 &= 0.
\end{align*}
If $\n \neq 0$, the first equation yields $z_1^2 = \bz_1^2$,
which says that $z_1^2$ is real, or equivalently,
$z_1$ is either purely real or 
purely imaginary. The second equation is then 
\begin{align*}
2 z_1^2 &= - e^{2 i \n} - e^{-2 i \n} |z_1|^4\\
& = - \cos(2 \n) - |z_1|^4 \cos (2 \n)
+ i\big(\!- \sin (2 \n) + |z_1|^4 \sin(2\n)\big).
\end{align*}  
Since the left hand side is real, and $\n\ne0$, 
we  conclude that that $|z_1|^4 = 1$. 
But then the equation reads $2 z_1^2 = - 2 \cos(2 \n)$,
so taking norms yields $1 = |z_1|^2 = \cos(2\n)$;
since $0 < \n < \pi/2$, this is a contradiction.
We conclude that $(\l, \l, 0)$ is the only possibility. 
\end{proof}

\begin{proposition}
\label{caseib}
If $\l = \m\ne0$ and $\n = 0$, then Case (i) happens over exactly two
points, which are those over which the quadric contains the entire
fiber.
\end{proposition} 
\begin{proof}
The first equation (\ref{c11}) is trivially satisfied, so we are just
looking at points in the discriminant locus 
of the form $(z_1, 0)$. The second equation (\ref{c12})
is
\begin{align*}
1 + |z_1|^4 + z_1^2 + \bz_1^2 = 0.
\end{align*}
Expanding, this becomes
\begin{align*}
1 + 2x_1^2 + x_1^4 - 2 y_1^2 + 2 x_1^2 y_1^2 + y_1^4
= (1 + x_1^2 - 2y_1 + y_1^2)(1 + x_1^2 + 2 y_1 + y_1^2). 
\end{align*}
It is easy to check this has two solutions $(x,y) = (0,\pm 1)$, which
are exactly the points over which the quadric contains the entire
fiber.
\end{proof}
\begin{remark} 
  Below we shall see a direct description of the situation described
  in Proposition~\ref{caseib}. After a conformal transformation, the
  discriminant locus will be a cone in $\rr^4$, with these singular
  points corresponding to the cone points at the origin and at
  infinity.
\end{remark}
\subsection{Case (ii) and a limit of Clifford tori}
Now assume that $z_1 = 0$ and $z_2 \ne0$. 
In this case we again have just one equation
\begin{align}
\label{badp3}
 \Big| c_3 c_4 z_2 |z_2|^2 + c_1 c_3 \bz_2 \Big| 
& =  \Big| \bc_3 \bc_4 z_2 |z_2|^2 + \bc_2 \bc_4 \bz_2 \Big|.
\end{align}
Again, we must also have the discriminant vanish
\begin{align}
\label{discp3}
0 = \Delta &= -(c_1 + c_4 z_2^2)(c_2 + c_3 \bz_2^2)
= - c_1 c_2 - c_2 c_4 z_2^2 - c_1 c_3 \bz_2^2 - c_3 c_4 |z_2|^4.
\end{align}
Recalling (\ref{c1c2c3c4}), 
the equations (\ref{badp3}) and (\ref{discp3}) may be rewritten as
\begin{align*}
( e^{ -2 i \n + \l + \m}  -   e^{ 2 i \n - \l - \m}) z_2^2 + 
( e^{ +2 i \n + \l + \m} -  e^{ -2 i \n - \l - \m} ) \bz_2^2 &= 
e^{ -2 ( \l + \m)} - e^{2( \l + \m)}\\
e^{2 i \n} + e^{- \l - \m} z_2^2 + e^{\l + \m} \bz_2^2
+ e^{-2 i \n} |z_2|^4 &= 0.
\end{align*}
\begin{proposition}
If Case (ii) occurs then $(\l,\m,\n) = (0,0,0)$. 
\end{proposition}
\begin{proof}
The proof is more or less the same as in Proposition \ref{casei}, 
with $\l - \mu$ replaced with $\l + \m$.
We conclude that $\l = - \m$ and $\n = 0$.
This implies the statement since $0 \le \l \le \m$.
\end{proof}

We already know that if $\l,\m,\n$ all vanish then $\Q$ is equivalent
to a real quadric, so that $D$ is a circle and coincides with $D_0$.
We next consider another important case in which we can describe $D$
explicitly.

\begin{proposition}\label{Cliff}
  If $\l = \m = 0$ and $\n\ne0$, then the discriminant locus is a
  smooth Clifford torus.
\end{proposition}
\begin{proof}
The discriminant equation in this case is 
\begin{align}\label{zzzz}
0 = e^{ 2i \n} + e^{- 2 i \n} |\zz|^4 + z_1^2
+ \bz_1^2 + z_2^2 + \bz_2^2.
\end{align}
The imaginary part of this equation is 
\begin{align*} 
\sin(2 \n) (1 - |\zz|^4). 
\end{align*}
Since $\n \neq 0$, we must have $|\zz|^4 = 1$.  That is, the
discriminant locus $D$ is a subset of the unit sphere.  The real part
of \rf{zzzz} is
\begin{align*}
(1 + |\zz|^4) \cos( 2 \n) + 2 ( x_1^2 - y_1^2) + 2 (x_2^2 - y_2^2).
\end{align*}
So $D$ is the subset of the unit sphere defined by 
\begin{align*}
y_1^2 + y_2^2 = x_1^2 + x_2^2 + \cos(2 \nu). 
\end{align*}
If we let $\xx = (x_1, x_2)$ and $\yy = (y_1, -y_2)$ (the minus sign
will be convenient in the next subsection), then the equations are
\begin{align}
\label{xxyy}
|\yy|^2 + |\xx|^2 = 1, \ \ \ \ |\yy|^2 - |\xx|^2 = \cos(2\n), 
\end{align}
from which it follows that 
\begin{align*}
|\yy|^2 &= \frac{ 1 + \cos(2 \n)}{2}, \ \ \ \ |\xx|^2 = \frac{ 1 - \cos(2 \n)}{2}, 
\end{align*}
which is clearly a smooth Clifford torus. 
\end{proof}
\begin{remark}
  For this family of tori, we see clearly from \rf{xxyy} how the torus
  limits to a circle as $\n \rightarrow 0$.
\end{remark}

\subsection{Case (iii) cannot happen} We turn to the generic situation
in which $z_1 \ne0$ and $z_2 \ne0$. 
Recalling (\ref{c1c2c3c4}), we rewrite the system as
\begin{align}\label{cansys}
\begin{split}
z_1 |\zz|^2( e^{-2 i \nu} - e^{i \theta} e^{2 i \nu})
& = \bz_1 ( e^{i \theta} e^{ - \lambda + \mu}  - e^{\lambda - \mu} )\\
z_2 |\zz|^2( e^{ -2 i \nu}  - e^{i \theta} e^{2 i \nu})
& = \bz_2 ( e^{i \theta} e^{ - \lambda - \mu} - e^{\lambda + \mu} ).
\end{split}
\end{align}

\begin{lemma} Let $\lambda\ne0$. A point $(z_1, z_2)$ on the 
discriminant locus with both $z_1$ and $z_2$ non-zero is a smooth
point of the discriminant locus.
\end{lemma}

\begin{proof}
  Since $z_1$ and $z_2$ are non-zero, taking norms of these 
equations gives
\begin{align*}
|\zz|^2 = \frac{  \big| e^{i \theta} e^{ - \lambda + \mu}  - e^{\lambda - \mu} \big|}
{ \big| e^{-2 i \nu} - e^{i \theta} e^{2 i \nu} \big|}
= \frac{ \big|  e^{i \theta} e^{ - \lambda - \mu} - e^{\lambda + \mu} \big|}
{ \big| e^{ -2 i \nu}  - e^{i \theta} e^{2 i \nu} \big|}.
\end{align*}
The denominators are the same, so this can only be possible if 
\begin{align*}
 e^{2( -\l + \m)} - e^{i \theta} - e^{-i \theta} + e^{2 ( \l - \m)}
= e^{ -2( \l + \m)} - e^{i \theta} - e^{-i \theta} + e^{2(\l + \m)},
\end{align*}
which says that $\cosh ( \l - \mu) = \cosh( \l + \m)$. 
Since $0 \le \l \le \mu$, we deduce that $\m - \l = \l + \m$, whence
$\l = 0$.
\end{proof}

Finally, we shall show that the assumption $\lambda\ne0$ is not needed
in the previous lemma.

\begin{proposition} If $0=\lambda<\mu$, then the discriminant locus is
  necessarily smooth.
\end{proposition}

\begin{proof} Fix $\zz=(z_1,z_2)=(x_1+iy_1,\>x_2+iy_2)$, and again set
  $\xx=(x_1,x_2)$, $\yy=(y_1,-y_2)$.

  The points where the gradient of $\Delta$ degenerates are given by
  \rf{cansys}. Rearranging, \[\ba{rcl} e^{-2i\n}z_1|\zz|^2 + e^{-\mu}\bz_1
  &=& e^{i\theta}(e^{2i\n}z_1|\zz|^2 + e^\mu\,\bz_1)\\[3pt]
  e^{-2i\n}z_2|\zz|^2 + e^\mu\,\bz_2 &=&
  e^{i\theta}(e^{2i\n}z_2|\zz|^2 + e^{-\mu}\bz_2).\ea\] Cross
  multiplying yields \be{cm}
  2\big(\cos(2\n)sX+\sin(2\n)cY\big)|\zz|^2+\sinh(2\mu)\bz_1\bz_2=0,\ee
  where $X+iY=\bz_1z_2$ and $c=\cosh\mu$, $s=\sinh\mu$. Therefore,
  \be{J} \Im(z_1z_2)=\xx\cdot(J\yy) = 0,\ee where $J\yy=(y_2,y_1)$.
  Thus $\yy=\mathbf0$ or \be{k} \xx=k\yy,\ee where $k\in\R$.

  The discriminant locus is given by \be{reD} 2c(|\xx|^2 - |\yy|^2) +
  (1 + |\zz|^4)\cos(2\nu)= 0,\ee \be{imD}\kern35pt
  4s\,\xx\cdot\yy-(|\zz|^4-1)\sin(2\nu)=0.\ee We know that these
  equations, with $c>1$, imply that $|\yy|>0$, so we may asssume
  \rf{k}.

We now have $|\zz|^2=(k^2+1)|\yy|^2$. Recalling the definition of
$\yy$, we also have
\[ \bz_1z_2=-(k-i)^2y_1y_2,\qquad z_1z_2=-(k^2+1)y_1y_2.\] Thus
$X=(1-k^2)y_1y_2$ and $Y=2ky_1y_2$. However, \rf{cm} implies that
\[\big(\cos(2\n)sX+\sin(2\n)cY\big)|\zz|^2=-cs\Re(z_1z_2),\] so either
$y_1y_2=0$ or \be{ss}
\left((1-k^2)s\cos(2\n)+2kc\sin(2\n)\right)|\yy|^2=cs.\ee But if
$y_1y_2=0$ then one of $z_1,z_2$ vanishes, which is impossible with
the current hypothesis $\l\ne\mu$. From \rf{reD} and \rf{imD}, we
also have \be{rek} 2c(k^2-1)|\yy|^2 + \big((k^2+1)^2|\yy|^4 +
1\big)\cos(2\nu)= 0,\ee \be{imk}\kern35pt 4sk|\yy|^2 -
\big((k^2+1)^2|\yy|^4-1\big)\sin(2\nu)=0.\ee

Consider first the special case $\n=0$. Then \rf{imk} gives $k=0$ and
thus $\xx=0$. From \rf{ss}, we have $|\yy|^2=c$. But \rf{rek} then
implies that $c=1$, which is not the case. 

Write $C=\cos(2\n)$ and $S=\sin(2\n)$. We may now assume that $s>0$
and $S>0$. In the new notation, \rf{ss} is \be{cs}
-sC\,k^2+2cS\,k+sC=cs/|\yy|^2.\ee We get an analogous expression by
multiplying \rf{rek} by $S$, \rf{imk} by $C$, and adding: \be{CS}
cS\,k^2+2sC\,k-cS=-CS/|\yy|^2.\ee This is surprising, since \rf{cs}
arises from the \emph{gradient} of $\Delta$, whereas \rf{CS} arises
from $\Delta$ itself.

Now multiply \rf{cs} by $cS$ and \rf{CS} by $sC$. Adding and
simplifying, we get
\be {ky} 2k|\yy|^2=sS.\ee
A similar operation involving subtraction gives
\be{kk} (1-k^2)|\yy|^2=cC.\ee 
We also obtain a solution
\[ 2|\yy|^2 = cC\pm\sqrt{c^2C^2+s^2S^2}=cC+\sqrt{s^2+C^2}.\]
Rearranging \rf{kk} gives
\[ (k^2+1)|\yy|^2=2|\yy|^2-cC=\sqrt{s^2+C^2}.\]
This can now be substituted into \rf{imk} to give
\[ 4sk|\yy|^2-(s^2+C^2-1)S=0.\]
Using also \rf{ky},
\[2s^2S-(s^2+C^2-1)S=0,\] and $c^2=C^2$, which contradicts
$c>1$.\end{proof}

We summarize what has been proved so far: 
\begin{theorem}
\label{dstruc}
Assume that $\Q$ is a quadric in the canonical form $(\l, \mu , \n)$
with $0 \le \l \le \mu$ and $0 \le \n < \pi /2$. Then the discriminant
locus $D$ is a smooth submanifold of real dimension $2$, unless $\l = \mu$
and $\n = 0$ (in which case there are exactly two singular points).
If $(\l,\mu,\n) = (0,0,\n)$ with $\n\ne0$ then $D$ is a smooth
Clifford torus. If $(\l,\mu,\n) = (0,0,0)$, then $D$ is a circle.
\end{theorem}

It will be the purpose of the next section to prove that $D$ is an
unknotted torus \emph{whenever} it is smooth and $(\l,\mu,\n)\ne(0,0,0)$.

\section{Topology of the discriminant locus}\label{topology}
We begin with the equation of the discriminant (\ref{ABC3}), 
which expands to 
\begin{align*}
\Delta = - e^{2 i \n} -  e^{-2 i \n} |\zz|^4
- e^{\l + \m} \bz_2^2 - e^{-\l - \m}z_2^2
- e^{\l - \m} \bz_1^2 - e^{\m - \l} z_1^2,
\end{align*} 
where $|\zz|^2 = |z_1|^2 + |z_2|^2$. We shall consider separately 
the real and imaginary parts
\begin{align}
\nonumber
\ \Re\Delta &= 2 \cosh( \l - \m) ( x_1^2 - y_1^2)
+ 2 \cosh( \l + \m) ( x_2^2 - y_2^2) + (1 + |\zz|^4) \cos( 2 \n),\\
\label{ImDelta}
\ \Im\Delta &= -4 \sinh( \l - \m) x_1 y_1 - 4 \sinh( \l + \m) x_2 y_2
- ( |\zz|^4 - 1) \sin(2 \n),
\end{align}
concentrating attention on the latter.

\begin{proposition}
\label{s3prop}
  If $ \nu \ne0$, then the zero set $\{\Im\Delta = 0 \}$ is a smoothly
  embedded 3-sphere in $\rr^4$.
\end{proposition}
\begin{proof}
  We shall show that \rf{ImDelta} determines a graph over the unit sphere
  $S^3(1)$. Take $(x_1, y_1, x_2, y_2) \in S^3(1)$, and consider the
  ray $r (x_1, y_1, x_2, y_2)$ for $0 < r < \infty$.  Along this ray,
  we have
\begin{align*}
\sin(2 \n) r^4  + r^2 \big(  4 \sinh( \l - \m) x_1 y_1  + 4 \sinh( \l + \m) x_2 y_2
\big) - \sin (2 \n) = 0.
\end{align*}
The quadratic formula gives the solution
\begin{align*}
r^2 = \frac{ - B \pm \sqrt{  B^2 + 4 \sin^2{2 \n} }} {2 \sin(2 \n)}, 
\end{align*}
where $B \equiv  4 \sinh( \l - \m) x_1 y_1  + 4 \sinh( \l + \m) x_2 y_2$.
By assumption, the discriminant is always positive. Moreover, this
equation always has exactly one positive root and the root can be
chosen smoothly, given that $\nu\ne0$. Therefore $\{ \Im\Delta= 0
\}$ is a smooth graph over $S^3$.
\end{proof}
\begin{remark} For $\n= 0$ then  $\{ \Im\Delta= 0 \}$
is clearly a real cone, which has a singularity at the origin. 
\end{remark}

Our next result extends Proposition~\ref{Cliff}.
\begin{theorem} 
\label{unknot}
If $\l<\m$, and $\n\ne0$, then $D$ is a smooth unknotted torus in
$\rr^4$.
\end{theorem}
\begin{proof}
  The assumption $\l<\m$ guarantees, from Theorem \ref{dstruc}, that
  $D$ is a submanifold. The Riemann-Hurwitz formula for the branched
  covering $\Q\to S^4$ is
\begin{align}
\label{Euler}
\chi(S^2 \times S^2) = 2 \chi(S^4) - \chi(D).
\end{align}
This yields $\chi(D) = 0$.

We first claim that no component of $D$ can be homeomorphic to $S^2$.
Assume by contradiction that $D$ has a component $D_s$ which is a
sphere. By translating $D_s$ in the normal direction to $\{\Im\Delta
= 0\}$ (which we know is a smoothly embedded $S^3$), we may assume
that $D_s$ is contained in a parallel $S^3$. By the Generalized
Schoenflies Theorem \cite[Theorem IV.19.11]{Bredon}, $D_s$ will bound
a $3$-disc $M$ in the parallel $S^3$.  Since $M$ is clearly disjoint
from the other components of $D$, the lift of $M \cup D_s$ under the
branched covering is a hypersurface in $S^2 \times S^2$, which
disconnects since $H^1(S^2 \times S^2) = 0$, as follows from the
Generalized Jordan Curve Theorem \cite[Theorem VI.8.8]{Bredon}.
Therefore the branched covering is trivial away from $D$, so there
cannot be any other branching components. This contradicts $\chi(D) =
0$.

Since there are no $S^2$ components, and $\chi(D) = 0$, all the
components must be tori.  Finally, we claim that there can only be one
component, and it is an unknotted torus.  To see this, take any torus
component $T^2$.  As above, push this torus off of $\{\Im\Delta= 0\}$
in a normal direction to another parallel $S^3$.  Since $T^2 \subset
S^3$, Alexander's Solid Torus Theorem says that $T^2$ must bound a
solid torus $S^1 \times D^2$ in $S^3$ \cite[page 107]{Rolfsen}. So we
have $T^2$ bounding a solid torus, which is clealy disjoint from the
other components of $D$ (the other components lie in the original
$S^3$, while the solid torus lies in a parallel $S^3$). The covering
argument above then shows there cannot be any other components of $D$.
Finally, a torus which bounds a solid torus in $\R^4$ must be
unknotted, it is isotopic to a standard torus $T^2 \subset\R^3
\subset\R^4 \subset S^4$ \cite{HosokawaKawauchi}.
\end{proof}

\subsection{Special cases revisited} We are now in a position to
refine the descriptions given in the previous section for cases in
which one of $\n$ or $\l$ vanishes.

First, we return to Case~(i) in which $\l=\m$ and $\nu=0$.

\begin{proposition} 
\label{2pts}
Given a quadric $\Q$ defined by the matrix $\Q_{( \l, \l, 0)}$ with
$\l\ne0$, the discriminant locus is a torus pinched at two points.
\end{proposition}
\begin{proof}
According to
\rf{trans}, this case equivalent to the case $\l=-\m\ne0$ and $\n=0$.
  Using \rf{tau} and \rf{nu-nu}, we see that this case corresponds to
  $x = y = v= 0$, and $u \in \rr$, with $u\ne\pm 1$. This
  means that $\Q$ is conformally equivalent to the zero set of the
  quadratic form
\begin{align*}
q = 2(u-1) \xa W_2 + 2(u+1) \xb W_1.
\end{align*}
To find its discriminant locus, we look at
\begin{align*}
\begin{split}
0 &=  2(u-1) \xa (\Wb) + 2( u +1) \xb (\Wa).
\end{split}
\end{align*}
Expanding this, we find $0 = A \xa^2 + 2B \xa \xb + C \xb^2$,
where
\begin{align*}
A &= 2( u -1) z_2, \ \ B = (u-1) \bz_1 + (u +1) z_1, \ \
C = - 2(u +1) \bz_2. 
\end{align*}
Therefore $D$ is defined by  
\begin{align*}
0 = B^2 - A C 
= (u-1)^2 \bz_1^2 + (u+1)^2 z_1^2 + (u^2-1) |z_1|^2 + 4(u^2 -1) |z_2|^2,
\end{align*}
with $u\ne\pm1$ constant.
Expanding into real and imaginary parts,
\begin{align*}
4 u^2 x_1^2 - 4 y_1^2 + 4 (u^2 - 1) (x_2^2 + y_2^2) = 0, \ \ \ \ 8 u x_1 y_1 = 0.
\end{align*}  
If $x_1 = 0 $, then $4 y_1^2 = 4(u^2 - 1)(x_2^2+y_2^2)$,
which is a cone for $|u| > 1$, and 
empty for $|u| < 1$. 
If $y_1 = 0$ then $4 u^2 x_1^2 = - 4 (u^2 - 1) (x_2^2 + y_2^2)$,
which is a cone for $|u| < 1$, and empty for $|u| > 1$.  
Performing an inversion as in (\ref{invlift}), we see a similar 
cone singularity at infinity.
When viewed as a subset of $S^3 \subset S^4$, $D$
is then clearly a torus pinched at two points.
\end{proof}

Next, we settle the case $\l<\m$ and $\n=0$, in which the discriminant
locus $D$ is smooth yet $\{\Im\Delta=0\}$ is not.

\begin{proposition}
\label{n0}
In the case of a quadric $\Q$ defined by $Q_{(\l,\m,0)}$ with
$\l<\m$, the discriminant locus is a smooth unknotted torus.
\end{proposition}
\begin{proof}
  In this case, by Theorem \ref{dstruc}, we know $D$ is a submanifold.
  Take a path $(\l, \m, \n_t)$ with $t \in [0,1]$ and $\n_t \ne0$ such
  that $(\l, \m, \n) \rightarrow (\l, \m, 0)$ as $t\to1$.  From
  Theorem \ref{unknot}, we know the corresponding $D_t$ are unknotted
  tori.  Also, since this is a smooth path of polynomials equations,
  their zero sets converge to $D$ in the Hausdorff distance.  Since
  $D_t$ is connected, and $D$ is a submanifold, $D$ must therefore be
  connected.  By the Euler characteristic formula (\ref{Euler}), $D$
  must be a torus. Finally,  the unknottedness follows
  since $D$ is the limit of smooth unknotted tori.
\end{proof}

\subsection{The non-diagonalizable case}
Recall from (\ref{ndABC}) above, the quadric is given by 
$A\xi^2+2B\xi+C=0$, with
\begin{align}
\label{ndABC'}
A=i + k \kern1pt\bz_1- \bz_2,\quad B=\Re(k z_2)+i \Im(z_1),
\quad C=i - k z_1- z_2.
\end{align}
A computation shows that the discriminant is
\begin{align*}
B^2 - AC  &= (k^2 -1) x_2^2 - y_2^2 + k^2 x_1^2
+ (k^2 -1) y_1^2 - 2 k y_1 + 1 + 2i ( x_2 + k x_1 y_2). 
\end{align*}
We let $f = \Re\Delta, g = \Im\Delta$,
and consider the matrix 
\begin{align}
J &= \left(
\begin{matrix}
\partial_{x_1} f & \partial_{y_1} f &\partial_{x_2} f &\partial_{y_2}f \\
\partial_{x_1} g & \partial_{y_1} g &\partial_{x_2} g &\partial_{y_2}g \\
\end{matrix}
\right)
= \left(
\begin{matrix}
2 k^2 x_1   & 2(k^2 -1)y_1 - 2k  
&2(k^2 -1) x_2  & - 2y_2 \\
k y_2 & 0  & 1  & k x_1 \\
\end{matrix}
\right).
\end{align}
\begin{proposition}
\label{ndp}  
If $k \in [0,1)$, then the discriminant 
locus is smooth on $\rr^4$. As a subset of $S^4$, there is 
exactly one singular point (the point at infinity), 
and the discriminant locus is a singular torus pinched at one point.  
\end{proposition}
\begin{proof}
The discriminant locus $D = \{ f = 0 \} \cap \{ g = 0 \}$.
In the case $k = 0$, the equations
for $D$ simplify to 
\begin{align}
x_2^2 + y_2^2 + y_1^2 = 1, \ \ \ \ x_2 = 0,
\end{align}
which is  a smooth cylinder. Such a cylinder 
is exactly an unknotted torus pinched at one point when 
viewed as a subset of $S^4$. Next, for $k \in (0,1)$, 
we show that $J$ has rank $2$ at every finite point on $D$. 
To see this, take the subdeterminant corresponding to
the first and last columns:
\begin{align}
\left|
\begin{matrix} 
2 k^2 x_1 & - 2y_2\\
k y_2 &  k x_1 \\
\end{matrix}
\right| 
= 2 k^3 x_1^2 + 2 k y_2^2 = 0. 
\end{align}
Since $k \in (0,1)$, this implies that $x_1 = 0$ 
and $y_2 = 0$.
The equations for $D$ simplify to 
\begin{align*}
0 = f =  (k^2 -1) x_2^2 + (k^2 -1) y_1^2 - 2 k y_1 + 1, \ \ \ \ 0 = g =  x_2,
\end{align*} 
which implies
\begin{align}
\label{g1}
0 = f &=  (k^2 -1) y_1^2 - 2 k y_1 + 1.
\end{align} 
Next take the subdeterminant corresponding to
the second and third columns to get
\begin{align}
y_1 = \frac{k}{k^2 -1}. 
\end{align}
Substituting this into (\ref{g1}), we obtain
\begin{align}
0 = \frac{k^2}{k^2 -1} - 2 \frac{k^2}{k^2 -1} + 1 =  - \frac{k^2}{k^2 -1} + 1,
\end{align}
which has no solution.   
Therefore $D$ is always smooth in $\rr^4$. 

 To identify the global topology of $D$, we argue as follows. 
We can identify our non-singular $2$-quadric
with $\cp^1 \times \cp^1$ such that the
fiber over infinity corresponds to a $\cp^1$ 
in one of the factors. Thus, 
\begin{align}
Q_2 \setminus \cp^1 = \cp^1 \times \rr^2.
\end{align}
The Riemann-Hurwitz formula for a branched covering is 
\begin{align}
2 = \chi(  \cp^1 \times \rr^2) = 2 \chi( \rr^4) - \chi(D \setminus \{p_{\infty}\})
= 2 - \chi(D \setminus \{p_{\infty}\}), 
\end{align}
which implies that $\chi(D \setminus \{p_{\infty}\} )= 0$.
We claim that $ D \setminus \{p_{\infty}\}$ is connected. 
This follows from our work in the diagonalizable case:
we can approximate our quadric $Q$ by a sequence 
of generic diagonalizable quadrics $Q_i$, with the discriminant 
loci $D_i$ converging to $D$ in the Hausdorff 
sense, with smooth convergence away from the singular 
point. We have proved above that $D_i$ are smooth unknotted tori, 
which in particular are connected.  Thus, we may connect
any 2 points in $D \setminus \{p_{\infty} \}$ by a path 
which is a limit of paths in the tori $D_i$ each of which 
avoids the singular point of convergence.  
Connectedness, together with $\chi = 0$,
imply that $D \setminus \{p_{\infty}\} = S^1 \times \rr$.

To finish the argument, an analysis of the singularity is needed;
we just briefly outline the details here. Consider the inverted quadric 
(\ref{invnon}), so that 
the singular point is at the origin. One then examines 
the intersection $D \cap S(r)$, where $S(r)$ is a small 
sphere of radius $r$ centered at the origin. 
An elementary computation shows that this limits to two disjoint 
$S^1$s in $S^3$ as $r \rightarrow 0$. This implies 
that the singularity is a double cone point, thus $D$ globally 
has the topology of a torus pinched at one point. 
\end{proof}

Combining Theorems \ref{quadricdisc}, \ref{dstruc}, and \ref{unknot}, 
and Propositions \ref{2pts}, \ref{n0}, and \ref{ndp}, 
we obtain Theorem \ref{quadricst}.

\subsection{Lifting the discriminant locus} 

In this subsection, we give a topological ``explanation'' of Cases (0)
and (1) of Theorem~\ref{quadricst}, and prove Theorem \ref{torusq}. 

In Case (0), the discriminant locus lifts to 
\begin{align*}
 \pi^{-1}(S^1) = S^1 \times \cp^1 = S^1 \times S^2, 
\end{align*}
since an oriented $\cp^1$-bundle over $S^1$ is trivial. 
This is a 3-real-dimensional submanifold of $\Q$, which must disconnect 
into two components
\cite[Theorem VI.8.8]{Bredon}. Any non-degenerate quadric is
diffeomorphic to $S^2 \times S^2$.  So we see that $S^2 \times S^2
\setminus S^1 \times S^2$ is equal to two copies of $S^4 \setminus
S^1$. We have the well-known isomorphism 
\begin{align*}
S^4 \setminus S^1 = \rr^4 \setminus \rr = D^2 \times S^2. 
\end{align*}
So we have the identification $S^2 \times S^2
\setminus S^1 \times \cp^1$ with two copies of $ D^2 \times S^2$.

In Case $(1)$, the discriminant locus is a torus, $D = S^1 \times
S^1$.  Since $D = D_1$, the lift of $D$ is $\pi^{-1}(D) = S^1 \times
S^1$.  So we see that $S^2 \times S^2 \setminus S^1 \times S^1$ is a
double cover of $S^4 \setminus S^1 \times S^1$.  Now $D$ is unknotted,
which means that $D$ bounds a solid torus $S^1 \times D^2$. The lift $\pi^{-1} (S^1
\times D^2)$ is a solid torus glued to itself along the boundary
torus, which is easily seen to be
\begin{align*}
\pi^{-1} (S^1 \times D^2 \cup S^1 \times S^1) = S^1 \times S^2. 
\end{align*}
So, by analogy with (i), we have the identification $S^2 \times S^2
\setminus S^1 \times S^2$ is two copies of $S^4 \setminus 
\ol{S^1 \times D^2}$. This also shows that the 
diagonalizable quadrics induce two OCSes on $S^4$ minus a solid torus.
This completes the proof of Theorem \ref{torusq}.

\bibliography{SVref}

\end{document}